\documentclass[12pt,fleqn,twoside]{article}

\usepackage{paralist}
\usepackage{a4}
\usepackage{amsthm}

\usepackage{ifthen}
\usepackage{amsmath}
\usepackage{amssymb}
\usepackage{graphicx} 
\usepackage{rotating}
\usepackage{color}
\usepackage{lscape}

\usepackage{stmaryrd} 

\usepackage{textcomp} 

\swapnumbers
\theoremstyle{plain}
\newtheorem{theorem}{Theorem}[section]
\newtheorem{lemma}[theorem]{Lemma}
\newtheorem{proposition}[theorem]{Proposition}
\newtheorem{corollary}[theorem]{Corollary}

\theoremstyle{definition}
\newtheorem{definition}[theorem]{Definition}
\newtheorem{definitions}[theorem]{Definitions}
\newtheorem{example}[theorem]{Example}
\newtheorem{examples}[theorem]{Examples}

\newtheorem{remark}[theorem]{Remark}
\newtheorem{remarks}[theorem]{Remarks}

\newtheorem{notation}[theorem]{Notation}
\newtheorem{OpRel}[theorem]{Operations and Relations}
\newtheorem{PolInv}[theorem]{The Galois connection $\boldsymbol{\Pol-\Inv}$}
\newtheorem{EndgQuord}[theorem]{The Galois connection $\boldsymbol{\End-\gQuord}$}
\newtheorem{lattices}[theorem]{The lattices $\cK^{(m)}_{A}$}

\numberwithin{equation}{theorem} 

\newcommand{\New}[1]{\emph{#1}}

\DeclareMathOperator{\Sta}{Sta} 

\newcommand{\uclose}[1]{\overline{#1}} 

\DeclareMathOperator{\End}{End}
\DeclareMathOperator{\Con}{Con}
\DeclareMathOperator{\Pol}{Pol}
\DeclareMathOperator{\Inv}{Inv}
\DeclareMathOperator{\Op}{Op}
\newcommand{\Opa}[1][]{\Op^{(#1)}}
\DeclareMathOperator{\Rel}{Rel}
\newcommand{\Rela}[1][]{\Rel^{(#1)}}

\DeclareMathOperator{\Eq}{Eq} \DeclareMathOperator{\Equ}{Eq}
\DeclareMathOperator{\Quord}{Quord} 
\DeclareMathOperator{\gQuord}{gQuord} 
\newcommand{\gQuorda}[1][]{\gQuord^{(#1)}} 

\DeclareMathOperator{\pr}{pr}

\newcommand{\preserves}{\mbox{ $\triangleright$ }}

\newcommand{\id}{\mathsf{id}}

\DeclareMathOperator{\ar}{ar} 

\newcommand{\cK}{\mathcal{K}}

 \newcommand{\bv}{\mathbf{v}}
 \newcommand{\br}{\mathbf{r}}
 \newcommand{\bb}{\mathbf{b}}

\newcommand{\bc}{\boldsymbol{c}}

\newcommand{\N}{\mathbb{N}}
\newcommand{\Z}{\mathbb{Z}}

\newcommand{\rmI}{{\rm I}}
\newcommand{\rmII}{{\rm II}}
\newcommand{\rmIII}{{\rm III}}

\newcommand{\Sg}[2][]{\ensuremath{\langle #2 \rangle_{#1}}}

\newcommand{\rfl}{\mathsf{ref}}
\newcommand{\tra}{\mathsf{tra}}
\newcommand{\gqu}{\mathsf{gqu}}

\newcommand{\trl}[1]{\mathsf{trl}(#1)} 

\let\rho=\varrho
\let\epsilon=\varepsilon
\let\phi=\varphi
\let\kappa=\varkappa

\let \restrictionORIGINAL=\restriction
\renewcommand{\restriction}{\mathclose\restrictionORIGINAL}

\author{Danica Jakub{\'\i}kov\'a-Studenovsk\'a%
  \\ Institute of
  Mathematics\\ P.J. \v{S}af\'arik University\\ Ko\v sice
  (Slovakia)
\and Reinhard P\"oschel%
 \\Institute of Algebra\\TU Dresden (Germany)
\and S\'andor Radeleczki%
\\Institute of Mathematics\\ University of
    Miskolc (Hungary)}

\date{\footnotesize version July 4, 2023}

\title{Generalized quasiorders and\\ the Galois connection $\End-\gQuord$}

\begin{document}

\maketitle

\begin{abstract}
 Equivalence relations or, more general, quasiorders (i.e., reflexive
and transitive binary relations) $\rho$ have the property that an
$n$-ary operation $f$ preserves $\rho$, i.e., $f$ is a polymorphism of
$\rho$, if and only if each translation (i.e., unary polynomial function obtained from
$f$ by substituting constants) preserves $\rho$, i.e., it is an
endomorphism of $\rho$. We introduce a wider class of relations
-- called generalized quasiorders -- of arbitrary arities with the same
property. 
With these generalized quasiorders we can characterize all
algebras whose clone of term operations is determined by its
translations by the above property, what generalizes affine
complete algebras. The results are based on the characterization of
so-called u-closed monoids (i.e., the unary parts of clones with the
above property) as Galois closures of the Galois connection
$\End-\gQuord$, i.e., as endomorphism monoids of generalized
quasiorders. The minimal u-closed monoids are described explicitly.
\end{abstract}

\section*{Introduction}\label{introduction}

Equivalence relations $\rho$ have the
remarkable well-known property that an
$n$-ary operation $f$ preserves $\rho$ (i.e., $f$ is a polymorphism of
$\rho$) if and only if each translation, i.e., unary polynomial function obtained from
$f$ by substituting constants, preserves $\rho$ (i.e., is an
endomorphism of $\rho$). Checking the proof one
sees that symmetry is not necessary, thus the same property, called
$\Xi$ in this paper (see \ref{MX}), also holds for quasiorders,
 i.e., reflexive and transitive relations.

No further relations with property $\Xi$ were known and once we
came up with the interesting (for us) question, if there are other
relations (than quasiorders) which satisfy $\Xi$, we hoped to prove that
$\Xi(\rho)$ implies that $\rho$ has to be a quasiorder (or at least to
be ``constructible'' from quasiorders). This attempt failed, but a new
notion was born: \New{transitivity} of a relation with higher
arity. The next step was to investigate reflexive and transitive
$m$-ary relations which naturally are called \New{generalized
  quasiorders} for $m\geq 3$ (for $m=2$ they coincide with usual
(binary) quasiorders) and which all have the property $\Xi$
(Theorem~\ref{A1b}). Moreover, these generalized quasiorders are
more powerful than quasiorders or equivalence relations (see Remark~\ref{B3}) and therefore
allow finer investigations of the structure of algebras $(A,F)$.

The next challenging question was: are there further relations with property $\Xi$, other than
generalized quasiorders? The answer is ``yes, but not really'': there
are relations $\rho$ satisfying $\Xi(\rho)$ and not being a generalized
quasiorder (see Example in \ref{A3New}), but each
such relation $\rho$ is ``constructively equivalent'' to  generalized quasiorders in
the sense that they generate the same relational clone and therefore
can be expressed mutually by primitive positive formulas (Proposition~\ref{A3Xi}).

With the property $\Xi$ the clone $\Pol\rho$ of polymorphisms is completely
determined by the endomorphism monoid $M=\End\rho$. Changing the point
of view and starting with an arbitrary monoid $M\leq A^{A}$ of unary
mappings, one can ask for the set $M^{*}$ of all operations whose
translations belong to $M$. Then $\Xi(\rho)$ means
$\Pol\rho=(\End\rho)^{*}$ (for details see Section~\ref{sec0}), in
particular, $M^{*}$ is a clone. But in general, $M^{*}$ is only a
so-called preclone (counterexample~\ref{M3A}). This leads to the question \textsl{When $M^{*}$ is a
clone?} and to the notion of a
\New{u-closed monoid} (namely if $M^{*}$ is a clone).

These u-closed monoids play a crucial role in this paper. Their
characterization via generalized quasiorders, namely as Galois closed monoids (of the Galois connection
$\End-\gQuord$ introduced in Section~\ref{secB}), is one of the main
results (Theorem~\ref{A3}) from which the 
answer to all above questions more or less follows.

The paper is organized as follows. All needed notions and notation are
introduced in \textsl{Section~\ref{prelim}}. \textsl{Section~\ref{sec0}} deals with the property $\Xi$ and the
u-closure and clarifies the preclone structure of
$M^{*}$. \textsl{Section~\ref{secA}} is the stage for the main player of this
paper: the generalized quasiorders. In particular, Theorem~\ref{A1b}
proves the property $\Xi$ for them. As already mentioned, in
\textsl{Section~\ref{secB}} the Galois connection $\End-\gQuord$ and
the crucial role of u-closed monoids is considered. Moreover, the
behavior of the u-closure under taking products and
substructures is clarified. In \textsl{Section~\ref{secC}} we consider
the u-closure of concrete monoids $M\leq A^{A}$, in particular all
minimal u-closed monoids are determined (Theorem~\ref{B0min}). In
\textsl{Section~\ref{secD}} we collect some facts and problems for
further research. In particular we show how the notion of an affine
complete algebra can be generalized via generalized quasiorders.


\section{Preliminaries}\label{prelim}

In this section we introduce (or recall) all needed notions and
notation together with some results. Throughout the paper, $A$ is a finite, nonempty
set. $\N:=\{0,1,2,\dots\}$ ($N_{+}:=\N\setminus\{0\}$) denotes the set of
(positive) natural numbers.

\begin{OpRel}\label{OpRel}
  Let $\Opa[n](A)$ and $\Rela[n](A)$ denote the set of all $n$-ary
  operations $f:A^{n}\to A$ and $n$-ary relations $\rho\subseteq
  A^{n}$, $n\in\N_{+}$, respectively. Further, let
  $\Op(A)=\bigcup_{n\in \N_{+}}\Opa[n](A)$ and  
   $\Rel(A)=\bigcup_{n\in \N_{+}}\Rela[n](A)$. 

The so-called \New{projections}
$e^{n}_{i}\in\Opa[n](A)$ are defined by
$e^{n}_{i}(x_{1},\dots,x_{n}):=x_{i}$ ($i\in\{1,\dots,n\}$,
$n\in\N_{+}$). The identity mapping is denoted by $\id_{A}$
($=e^{1}_{1}$).

$C:=\{\bc_{a}\mid a\in  A\}$ is the set of all \New{constants}, 
  considered as unary
   operations given by $\bc_{a}(x):=a$ for $a\in A$. 

Special sets of relations are $\Equ(A)\subseteq\Quord(A)$ and $\Quord(A)\subseteq\Rela[2](A)$ of all
\New{equivalence relations} (reflexive, symmetric and transitive) and
\New{quasiorder relations} (reflexive and transitive), respectively,
on the set $A$.

For $f\in\Opa[n](A)$ and $r_{1},\dots,r_{n}\in A^{m}$,
$r_{j}=(r_{j}(1), \dots, r_{j}(m))$,
($n,m\in\N_{+}$, $j\in\{1,\dots,n\}$),
let $f(r_{1},\dots,r_{n})$ denote the $m$-tuple obtained from
componentwise application of $f$, i.e., the $m$-tuple
$(f(r_{1}(1),\dots,r_{n}(1)),\dots,f(r_{1}(m),\dots,r_{n}(m)))$.

For $f\in\Opa[n](A)$ and $g_{1},\dots,g_{n}\in\Opa[1](A)$, the
\New{composition} $f[g_{1},\dots,g_{n}]$ is the unary operation given by
$f[g_{1},\dots,g_{n}](x):=f(g_{1}(x),\dots,g_{n}(x))$, $x\in A$.
\end{OpRel}

\begin{PolInv}\label{PolInv} An operation $f\in\Opa[n](A)$
  \New{preserves} a relation $\rho\in\Rela[m](A)$ ($n,m\in\N_{+}$) if
for all $r_{1},\dots,r_{n}\in\rho$ we have
  $f(r_{1},\dots,r_{n})\in\rho$, notation $f\preserves\rho$.

 The
Galois connection induced by $\preserves$ gives rise to
several operators as follows. For $Q\subseteq\Rel(A)$ and
$F\subseteq\Op(A)$ let
\begin{align*}
  \Pol Q&:=\{f\in\Op(A)\mid \forall \rho\in Q: f\preserves\rho\}
          &\text{ (\New{polymorphisms}),}\\
  \Inv F&:=\{\rho\in\Rel(A)\mid \forall f\in F: f\preserves\rho\} 
&\text{ (\New{invariant relations}),}\\
  \End Q&:=\{f\in\Opa[1](A)\mid \forall \rho\in Q: f\preserves\rho\} 
&\text{ (\New{endomorphisms}),}\\
\Con F&:=\Con(A,F):=\Inv F\cap\Eq(A)&\text{ (\New{congruence relations}),}\\
\Quord F&:=\Quord(A,F):=\Inv F\cap\Quord(A)&\text{ (\New{compatible quasiorders}).}
\end{align*}
The Galois closures for $\Pol-\Inv$ and $\End-\Inv$ are known and can be
characterized as follows: $\Pol\Inv F=\Sg{F}$ (clone generated by
$F$), $\Inv \Pol Q=[Q]_{\exists,\land,=}$ (relational clone, generated
by $Q$, equivalently characterizable as closure with respect to
primitive positive formulas, i.e., formulas containing variable and
relational symbols and only $\exists,\land,=$), $\End\Inv M=\Sg{M}$
((sub)monoid generated by $M\subseteq A^{A}$), $\Inv\End
Q=[Q]_{\exists,\land,\lor,=}$ (weak Krasner algebra generated by $Q$,
equivalently characterizable as closure with respect to positive
formulas, i.e., formulas containing variable and
relational symbols and $\exists,\land,\lor, =$). We refer to, e.g.,
\cite[1.2.1, 1.2.3, 2.1.3]{PoeK79}, \cite{BodKKR69a}, \cite{Poe04a}, \cite{KerPS2014}.
\end{PolInv}

\begin{definition}\label{M5}

 A set $F \subseteq \Op(A)$ is called a \New{preclone} if it
  contains $\id_{A}$ and is closed under the operations $\zeta$,
  $\tau$ and $\circ$ that are defined as
  follows. Let $f\in\Opa[n](A)$ and $g \in \Opa[m](A)$, $n,m\in\N_{+}$. Then

  \begin{enumerate}[(1)]

  \item\label{M5-1}
    $\id_A(x):=x$ (\New{identity operation});

  \item \label{M5-2}
       $(\zeta f)(x_{1},x_{2},\dots,x_{n}) := f(x_{2},\dots,x_{n},x_{1})$ 
      (\New{cyclic shift}), if
    $n = 1$ then $\zeta f := f$;

  \item \label{M5-3}
      $(\tau f)(x_{1},x_{2},x_{3},\dots,x_{n}) 
               := f(x_{2},x_{1},x_{3},\dots,x_{n})$\\
      (\New{permuting the first two
      arguments}),
     if
    $n = 1$ then $\tau f := f$;

  \item \label{M5-4}
     $ (f\circ
    g)(x_{1},\dots,x_{m},x_{m+1},\dots,x_{m+n-1}):=
    f(g(x_{1},\dots,x_{m}),x_{m+1},\dots,x_{m+n-1})$
(\New{composition}).
    
  \end{enumerate}

For later use we introduce here also the operations $\nabla$
(\New{adding a fictitious argument} at first place) and $\Delta$
(\New{identification of the first two arguments}):
\begin{enumerate}[\rm(1)]
\setcounter{enumi}{4}
  \item \label{M5-5}
       $(\nabla f)(x_{1},x_{2},\dots,x_{n+1})
      :=f(x_{2},\dots,x_{n+1})$,

\item\label{M5-6} $(\Delta
  f)(x_{1},\dots,x_{n-1}):=f(x_{1},x_{1},\dots,x_{n-1})$ if $n\geq 2$,
  and $\Delta f=f$ for $n=1$.
\end{enumerate}
\end{definition}

Remarks: 
Clearly, because of \eqref{M5-1} and \eqref{M5-4}, the unary part
$F\cap\Opa[1](A)$ of a preclone $F$ is a monoid.
The
$(m+n-1)$-ary function $f \circ g$ (defined in \eqref{M5-4}) sometimes 
is called \New{linearized composition} (or \New{superposition}),
because this is a special case of the general \emph{linearized
  composition}, \emph{linearization} or \emph{superposition} mentioned in
\cite[2.1]{BruDPS93}, \cite[page~2]{GraW1984} or
\cite[Section~2.1]{Leh2010}, respectively.

Preclones, also known as operads, can be thought as ``clones where
identification of variables is not allowed'' (cf.~\ref{M5A}).
The term \emph{preclone} was introduced by \'Esik
and Weil \cite{EsiW2005} in a study of the syntactic properties of
recognizable sets of trees. 
A general characterization of preclones as Galois closures via
so-called matrix collections can be found in \cite{Leh2010}.
The notion of \emph{operad} originates from the work in algebraic
topology by May \cite{May1972} 
and Boardman and Vogt \cite{BoaV1973}.
For general background and basic properties of operads, we refer the
reader to the survey article by Markl \cite{Mar2008}.  

Clones are special preclones. There are many (equivalent) definitions
of a clone. One of these 
definitions is that a clone is a set $F\subseteq\Op(A)$ closed under
\ref{M5}\eqref{M5-1}-\eqref{M5-6}, \cite[1.1.2]{PoeK79}. Therefore we have:

\begin{lemma}\label{M5A}
A preclone is a clone if and only if it is also closed under $\nabla$
(adding ficticious variables) and $\Delta$
(identification of variables). 
\end{lemma}

For $F\subseteq \Op(A)$, the clone generated by $F$ is denoted
by $\Sg{F}$ or $\Sg[A]{F}$.

\section{The property $\boldsymbol \Xi$ and u-closed monoids}\label{sec0}

Equivalence relations or, more general, quasiorder relations $\rho$ have the
 remarkable property $\Xi$ (see \ref{MX} below) that a polymorphism $f\in\Pol\rho$ is
determined by its translations  $\trl{f}$  defined as
follows:

\begin{definitions}\label{M1}
For an $n$-ary operation $f:A^{n}\to A$, $i\in\{1,\dots,n\}$ and a tuple 
$\mathbf{a}=(a_{1},\dots,a_{i-1},a_{i+1},\dots,a_{n})\in A^{n-1}$, 
let $f_{\mathbf{a},i}$ be the unary polynomial function
\begin{align}
  f_{\mathbf{a},i}(x):=f(a_{1},\dots,a_{i-1},x,a_{i+1},\dots,a_{n}),\label{M1a}
\end{align}
called \New{translation} (see, e.g., \cite[Definition 1.4.7]{Ihr2003}, \New{1-translation} in \cite[p.~375]{Gra2008}, \New{basic
  translation} in \cite{Mal1963}) and let $\trl{f}$ 
be the set of all such \textbf{tr}ans\textbf{l}ations
$f_{\mathbf{a},i}$. For constants (as well as for arbitrary unary functions) $f$ we put $\trl{f}:=\{f\}$.
For $F\subseteq \Op(A)$ let 
\begin{align}\label{M1b}
  \trl{F}:=\bigcup_{f\in F}\trl{f}\subseteq A^{A}.
\end{align}
Given a set $M\subseteq A^{A}$ we define
\begin{align}\label{M1c}
  M^{*}:=\{f\in \Op(A)\mid \trl{f}\subseteq M\}.
\end{align}
\end{definitions}

Remark: For a unary function $f$ we have
$f\in (\trl{f})^{*}$, in particular
$f\in M^{*}$ implies $f\in M$. Thus $\trl{M^{*}}=M$ for every 
$M\subseteq A^{A}$.

\begin{definition}[\textbf{The property $\boldsymbol\Xi$}]\label{MX}
For a relation $\rho\in\Rel(A)$ we consider the following property
$\Xi$ in three equivalent formulations:
\begin{align}\label{MX1}
  \Xi(\rho):&\iff\quad \forall f\in\Op(A):\;
  f\preserves\rho\iff \trl{f}\preserves\rho\\\label{MX2}
  &\iff \quad \forall f\in\Op(A):\;
  f\in\Pol\rho\iff \trl{f}\subseteq\End\rho\\\label{MX3}
  &\iff \quad 
  \Pol\rho=(\End\rho)^{*}.
\end{align}
This can be extended to sets $Q\subseteq\Rel(A)$ just by substituting
$Q$ for $\rho$ in the above definition, e.g., $\Xi(Q)\iff\Pol Q=(\End Q)^{*}$.
\end{definition}

\begin{remark}\label{M2X}
As noticed above, it is well-known that \emph{$\Xi(\rho)$ holds for $\rho\in\Equ(A)$ or, more
general, for $\rho\in\Quord(A)$}. Equivalently, expressed with the
usual notions of congruence or quasiorder lattices, this means 
\begin{align*}
  \Con(A,F)=\Con(A,\trl{F})\text{ and }\Quord(A,F)=\Quord(A,\trl{F})
\end{align*}
for each algebra $(A,F)$ ($F\subseteq \Op(A)$).

\end{remark}

Clearly, there arises the question already mentioned in the
introduction:

 \centerline{\textit{Does there exist other relations
$\rho$ with the property $\Xi(\rho)$?}}
Since $\Xi(\rho)$ implies that
$(\End\rho)^{*}$ is a clone and therefore it is closed under $\nabla$
(cf.~\ref{M5A}). As we shall see in \ref{M6} below this also implies
$C\subseteq \End\rho$, 
what expresses the fact that $\rho$ is reflexive (for definition
see~\ref{A1}). However, the converse is not true: not each reflexive
relation satisfies $\Xi(\rho)$ as the following example shows.

\begin{example}\label{M3A}
Let $A=\{0,1,2\}$ and $M:=\End\rho$ for the binary relation
$\rho=\{(0,0),(1,1), (2,2), (0,1), (1,2)\}$. Note that $\rho$ is
reflexive but not transitive. Define $f:A^{2}\to A$ by the following
table:
\begin{center}
\begin{tabular}[c]{|r|rcc|}
\hline
  $f(x,y)$ & $y$= 0&1&2\\\hline
 $x$= 0 & 0&0&1\\
  1 & 0&0&1\\
  2 & 1&1&2\\\hline
\end{tabular}
\end{center}
One can immediately check that each unary polynomial $f_{\mathbf{a},i}$ 
preserves $\rho$, i.e., $\trl{f}\subseteq M$, but
$g:=\Delta f$ (i.e., $g(x)=f(x,x)$) is the mapping $0\mapsto 0$, $1\mapsto 0$, $2\mapsto 2$
which does not belong to $M$ (since $g$ does not preserve $\rho$ because
$g$ maps $(1,2)\in\rho$ to $(0,2)\notin\rho$). Thus $f\in M^{*}$ but $g\notin M^{*}$. Hence
$M^{*}$ is not a clone. 

\end{example}

Since $M^{*}$ is not always a clone, there also arises the question: what is
the algebraic nature of the sets $M^{*}$? The answer gives the
following proposition.

\begin{proposition}\label{M6}
  Let $M\leq A^{A}$ be a monoid. Then $M^{*}$ is a preclone
  (i.e., it contains $\id_{A}$ and is closed under the
   operations $\zeta,\tau,\circ$, cf.~{\rm\ref{M5}}). Moreover, $M^{*}$ is
  closed under $\nabla$ (cf.~{\rm\ref{M5}\eqref{M5-5}}) if and only if $C\subseteq M$.
\end{proposition}

\begin{proof} Clearly $\id_{A}\in M\subseteq M^{*}$.
It is straightforward to check that for $f,g\in M^{*}$ also $\zeta f,
\tau f$ and $f\circ g$ belong to $M^{*}$ (notation see
\ref{M5}). We show it for \ref{M5}\eqref{M5-4}: if all variables
$x_{1},\dots,x_{m},\dots,x_{m+n-1}$, with exception of $x_{i}$, are constant, say
$\mathbf{a}=(a_{1},\dots,a_{m},\dots,a_{m+n-1})$, then, for $i\geq m+1$, we have
$(f\circ
g)_{\mathbf{a},i}=f(b,a_{m+1},\dots,x_{i},\dots,a_{m+n-1})$
with $b:=g(a_{1},\dots,a_{m})$,
what obviously belongs to $\trl{f}\subseteq M$. 
If $i\leq m$, then $g_{\mathbf{a'},i}\in M$ (because $g\in M^{*}$),
$\mathbf{a'}:=(a_{1},\dots,a_{m})$ (without the $i$-th component) and
$f_{\mathbf{a''},1}(x)=f(x,a_{m+1},\dots,a_{m+n-1})$ belongs to $M$ (because $f\in M^{*}$), where
$\mathbf{a''}:=(a_{m+1},\dots,a_{m+n-1})$, consequently $(f\circ
g)_{\mathbf{a},i}(x)=f_{\mathbf{a''},1}(g_{\mathbf{a'},i}(x))$ also
  belongs to the monoid $M$. Thus $\trl{f\circ g}\subseteq M$, i.e.,
  $f\circ g\in M^{*}$.

Further we observe $\nabla f=f\circ e^{2}_{2}$ and 
$e^{2}_{2}=\nabla\id_{A}$ where $e^{2}_{2}$  is 
the binary projection $e^{2}_{2}(x_{1},x_{2})=x_{2}$.
Thus
the preclone $M^{*}$ is
closed under $\nabla$ if and only if $e^{2}_{2}\in M^{*}$.
But
  $\trl{e^{2}_{2}}=\{\id_{A}\}\cup C$ (since
  $e^{2}_{2}(a,x)=\id_{A}(x)$ and $e^{2}_{2}(x,a)=\bc_{a}$ for
  $a\in A$), therefore
  $e^{2}_{2}\in M^{*}\iff\trl{e^{2}_{2}}\subseteq M\iff C\subseteq M$, and we are done.
\end{proof}

\begin{remark}\label{M6A}  $M^{*}$ is a preclone for a
  monoid $M$ by~\ref{M6}. Conversely, for a preclone $P$ the translations
  $\trl{P}$ form a monoid (because of \ref{M5}\eqref{M5-1} and \eqref{M5-4}).
Thus we can consider the following two mappings between monoids and
preclones:
\begin{align*}
 \phi&: P\mapsto \trl{P},\text{ where $P$ is a preclone on $A$.}\\
 \psi&: M\mapsto M^{*}, \phantom{xi}\text{where $M\leq A^{A}$ is a monoid on $A$}.
\end{align*} 

Then $(\phi,\psi)$ is a residuated
pair of mappings (covariant Galois connection) between the lattice of submonoids
of $A^{A}$ and the lattice of preclones on $A$. We have $\phi(P)\subseteq
M\iff P\subseteq \psi(M)$. Moreover, the corresponding kernel operator
$\phi(\psi(M))=\trl{M^{*}}=M$ is trivial (cf.\ remark in Definition~\ref{M1}).
However, the corresponding closure operator $P\mapsto \psi(\phi(P))$ is
nontrivial and it is an open problem which preclones are closed, i.e.,
when do we have $P=\psi(\phi(P))=(\trl{P})^{*}$? 
\end{remark}

\begin{lemma}\label{M3B} Let $M_{i}\le A^{A}$, $i\in I$. Then 
   $(\bigcap_{i\in I} M_{i})^{*}=\bigcap_{i\in I}M_{i}^{*}$.
\end{lemma}

\begin{proof}
Since, for a
residuated pair $(\phi,\psi)$, the residual $\psi$ is meet-preserving,
the Lemma immediately follows from~\ref{M6A}. We add a direct proof just using the definitions:
\begin{align*}
  f\in(\textstyle\bigcap_{i\in I} M_{i})^{*}
&\iff \trl{f}\subseteq\textstyle\bigcap_{i\in I}M_{i}
\iff \forall i\in I: \trl{f}\subseteq M_{i}\\
&\iff \forall i\in I: f\in M_{i}^{*}
\iff f\in \textstyle\bigcap_{i\in I}M_{i}^{*}. 
\end{align*}
\end{proof}

Since $M^{*}$ is not always a clone, the question arises:
For which monoids $M\leq A^{A}$ the preclone $M^{*}$ is a
clone?
To attack this problem we introduce the \New{u-closure} $\uclose{M}$
what shall lead to the equivalent problem
(cf.~\ref{M3D}\eqref{M3Diii}) of characterizing u-closed monoids.

\begin{definition}\label{M3C} For $M\subseteq A^{A}$ let
  \begin{align*}
    \uclose{M}:=\bigcap\{N\mid M\subseteq N \leq A^{A},\text{ and
      $N^{*}$ is a clone}\}.
  \end{align*}
A monoid $M\leq A^{A}$ is called \New{u-closed} if $\uclose{M}=M$.

\end{definition}

\begin{remarks}\label{M3D} Let $M\subseteq A^{A}$.
  \begin{enumerate}[\rm(i)]
  \item\label{M3Di} The operator $M\mapsto \uclose{M}$ is a closure operator
    (this follows from Lemma~\ref{M3B}).

  \item\label{M3Dii} $\uclose{M}$ is a monoid containing $C$ and
    $(\uclose{M})^{*}$ is a clone 
    (the latter follows from \ref{M3B} because, by
    definition, $\uclose{M}$ is the  
    intersection of \textsl{monoids} $N$ with $N^{*}$ being a clone; thus from
     \ref{M6} follows $C\subseteq \uclose{M}$, too). In particular
    we have $\uclose{\Sg{M}}=\uclose{M}=\Sg{\uclose{M}}$.

  \item\label{M3Diii}  \textit{$M$ is u-closed  (i.e. $\uclose{M}=M$)
      if and only if $M^{*}$ 
    is a clone} (in fact, ``$\Rightarrow$'' follows from \eqref{M3Dii},
    ``$\Leftarrow$'' follows from definition \ref{M3C}).

  \end{enumerate}
\end{remarks}

 A characterization of u-closed
monoids $M$ will be given in the next sections (Proposition~\ref{M7},
Theorem~\ref{A3} and Corollary~\ref{A3cor}).

\section{Generalized quasiorders}\label{secA}

\begin{notation}\label{M5B}
Let $A=\{a_{1},\dots,a_{k}\}$ and $M\leq A^{A}$. We define the
following $|A|$-ary relation:
\begin{align*}
  \Gamma_{M}:=\{(ga_{1},\dots,ga_{k})\mid g\in M\}.
\end{align*}
Thus $\Gamma_{M}$ consists of all ``function tables''
$\br_{g}:=(ga_{1},\dots,ga_{k})$ (considered as elements
(columns) of a relation) of the unary
functions $g$ in $M$. 

In particular, we have 
\begin{align*}
  M=\End\Gamma_{M}.
\end{align*}
 In fact,
$h\in\End\Gamma_{M}$, i.e., $h\preserves\Gamma_{M}$, implies $h(\br_{\id})\in\Gamma_{M}$, i.e.,
$\exists g\in M: h(\br_{id})=\br_{g}$ what gives $h=g\in M$. Conversely ,
if $h\in M$, then $h(\br_{g})=\br_{h\circ g}\in \Gamma_{M}$ for all
$\br_{g}\in\Gamma_{M}$, i.e., $h\preserves \Gamma_{M}$.

Moreover, it is known that $\Pol\Gamma_{M}$ coincides with the so-called
stabilizer $\Sta(M)$ of $M$ and it is the largest element in the
monoidal interval defined by $M$ (all clones with unary part $M$ form
an interval in the clone lattice, called \New{monoidal
  interval}, cf., e.g., \cite[Proposition 3.1]{Sze86}).
If $F$ is a clone with $F^{(1)}=M$, then
$\Gamma_{M}$ is the so-called first graphic of $F$ denoted by
$\Gamma_{F}(\chi_{1})$ in \cite{PoeK79}.

\end{notation}

\begin{definition}[\textbf{generalized quasiorder}]\label{A1}
An $m$-ary relation $\rho\subseteq A^{m}$ is called  \New{reflexive}
if $(a,\dots,a)\in\rho$ for all $a\in A$, and it is called
\New{(generalized) transitive} if for every $m\times m$-matrix
$(a_{ij})\in A^{m\times m}$ we have:
if every row and every column belongs to $\rho$ -- for this property we
write $\rho\models (a_{ij})$ -- then also the diagonal
$(a_{11},\dots,a_{mm})$ belongs to $\rho$, cf.~Figure~\ref{Fig1}.

\begin{figure}
%
\begin{picture}(0,0)%
\includegraphics{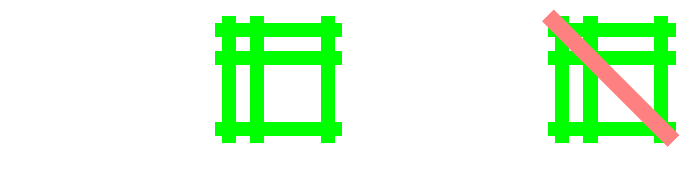}%
\end{picture}%
\setlength{\unitlength}{4972sp}%
\begingroup\makeatletter\ifx\SetFigFont\undefined%
\gdef\SetFigFont#1#2#3#4#5{%
  \reset@font\fontsize{#1}{#2pt}%
  \fontfamily{#3}\fontseries{#4}\fontshape{#5}%
  \selectfont}%
\fi\endgroup%
\begin{picture}(4365,1086)(-2849,-1331)
\put(-854,-466){\makebox(0,0)[lb]{\smash{{\SetFigFont{12}{14.4}{\rmdefault}{\mddefault}{\updefault}{\color[rgb]{0,0,0}$a_{1m}$}%
}}}}
\put(-1484,-466){\makebox(0,0)[lb]{\smash{{\SetFigFont{12}{14.4}{\rmdefault}{\mddefault}{\updefault}{\color[rgb]{0,0,0}$a_{11}$}%
}}}}
\put(-854,-1096){\makebox(0,0)[lb]{\smash{{\SetFigFont{12}{14.4}{\rmdefault}{\mddefault}{\updefault}{\color[rgb]{0,0,0}$a_{mm}$}%
}}}}
\put(-1484,-1096){\makebox(0,0)[lb]{\smash{{\SetFigFont{12}{14.4}{\rmdefault}{\mddefault}{\updefault}{\color[rgb]{0,0,0}$a_{m1}$}%
}}}}
\put(-2834,-736){\makebox(0,0)[lb]{\smash{{\SetFigFont{14}{16.8}{\rmdefault}{\mddefault}{\updefault}$\rho\models (a_{ij}):\Longleftrightarrow$}}}}
\put(-539,-466){\makebox(0,0)[lb]{\smash{{\SetFigFont{12}{14.4}{\rmdefault}{\mddefault}{\updefault}{\color[rgb]{0,0,0}$\in\rho$}%
}}}}
\put(-539,-646){\makebox(0,0)[lb]{\smash{{\SetFigFont{12}{14.4}{\rmdefault}{\mddefault}{\updefault}{\color[rgb]{0,0,0}$\in\rho$}%
}}}}
\put(-539,-1096){\makebox(0,0)[lb]{\smash{{\SetFigFont{12}{14.4}{\rmdefault}{\mddefault}{\updefault}{\color[rgb]{0,0,0}$\in\rho$}%
}}}}
\put(-1529,-1276){\makebox(0,0)[lb]{\smash{{\SetFigFont{12}{14.4}{\rmdefault}{\mddefault}{\updefault}{\color[rgb]{0,0,0}$\in\!\rho$}%
}}}}
\put(-1259,-1276){\makebox(0,0)[lb]{\smash{{\SetFigFont{12}{14.4}{\rmdefault}{\mddefault}{\updefault}{\color[rgb]{0,0,0}$\in\!\rho$}%
}}}}
\put(-809,-1276){\makebox(0,0)[lb]{\smash{{\SetFigFont{12}{14.4}{\rmdefault}{\mddefault}{\updefault}{\color[rgb]{0,0,0}$\in\!\rho$}%
}}}}
\put(631,-466){\makebox(0,0)[lb]{\smash{{\SetFigFont{12}{14.4}{\rmdefault}{\mddefault}{\updefault}{\color[rgb]{0,0,0}$a_{11}$}%
}}}}
\put(1261,-1096){\makebox(0,0)[lb]{\smash{{\SetFigFont{12}{14.4}{\rmdefault}{\mddefault}{\updefault}{\color[rgb]{0,0,0}$a_{mm}$}%
}}}}
\put(631,-1096){\makebox(0,0)[lb]{\smash{{\SetFigFont{12}{14.4}{\rmdefault}{\mddefault}{\updefault}{\color[rgb]{0,0,0}$a_{m1}$}%
}}}}
\put(1261,-466){\makebox(0,0)[lb]{\smash{{\SetFigFont{12}{14.4}{\rmdefault}{\mddefault}{\updefault}{\color[rgb]{0,0,0}$a_{1m}$}%
}}}}
\put(-89,-781){\makebox(0,0)[lb]{\smash{{\SetFigFont{14}{16.8}{\rmdefault}{\mddefault}{\updefault}$\boldsymbol{\implies}$}}}}
\put(1441,-1231){\makebox(0,0)[lb]{\smash{{\SetFigFont{14}{16.8}{\rmdefault}{\mddefault}{\updefault}\rotatebox{-43}{$\boldsymbol{\in\rho}$}}}}}
\end{picture}%

\caption{Transitivity for an $m$-ary relation $\rho$}\label{Fig1}
\end{figure}
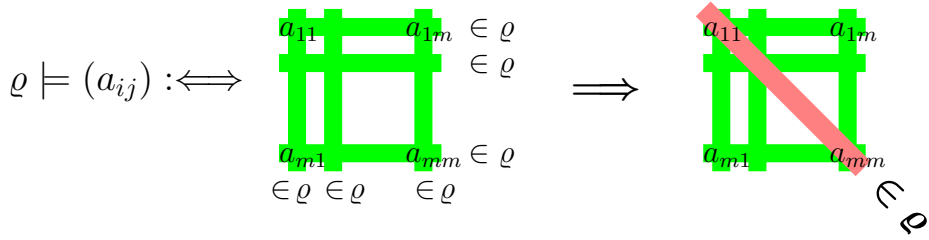

A reflexive and transitive $m$-ary relation is called \New{generalized
quasiorder}. The set of all generalized
quasiorders on the base set $A$ is denoted by $\gQuord(A)$, and
$\gQuorda[m](A):=\Rela[m](A)\cap\gQuord(A)$ will denote the $m$-ary
generalized quasiorders. 
\end{definition}

\begin{examples}\label{A1Ex} From the definitions easily follows:

(i) Each quasiorder
    (i.e., binary reflexive and transitive 
    relation) is also a generalized quasiorder. The converse is also
    true: Each binary generalized quasiorder is a usual quasiorder
    relation, i.e., we have $\gQuord^{(2)}(A)=\Quord(A)$.

(ii) Each so-called diagonal relation is a generalized
     quasiorder. Here an $m$-ary relation $\delta\in\Rel(A)$
     ($m\in\N_{+}$) is called 
     \New{diagonal relation} if there exists an equi\-valence relation
     $\epsilon$ on the set $\{1,\dots,m\}$ of indices such that 
     $\delta=\{(a_{1},\dots,a_{m})\in A^{m}\mid \forall
     i,j\in\{1,\dots,m\}: (i,j)\in\epsilon\implies a_{i}=a_{j}\}$.
\end{examples}

We generalize the notation $\rho\models(a_{ij})$ to
$n$-dimensional ``$m\times \ldots\times m$-matrices'' (tensors)
$(a_{i_{1},\dots, i_{n}})\in A^{m\times\ldots\times m}$ where $i_{1},\dots,i_{n}\in\{1,\dots,m\}$):
$\rho\models (a_{i_{1},\dots, i_{n}})$ denotes the fact that every
``row'' in each dimension belongs to $\rho$, i.e., for each index
$j\in\{1,\dots,n\}$ and any fixed
$i_{1},\dots,i_{j-1},i_{j+1},\dots,i_{n}$ the $m$-tuple
$a_{i_{1},\dots,[j],\dots,i_{n}}:=
      (a_{i_{1},\dots,1,\dots,i_{n}},\dots,a_{i_{1},\dots,m,\dots,i_{n}})$
(the indices $1,\dots,m$ are on the $j$-th place in the index
sequence) belongs to $\rho$.

\emph{Example:} 
For $n=3$, $\rho\models (a_{i_{1},i_{2},i_{3}})$ means that for all
$i_{1},i_{2},i_{3}\in\{1,\dots,m\}$ we have 
$(a_{1,i_{2},i_{3}},\dots,a_{m,i_{2},i_{3}})\in\rho$, 
$(a_{i_{1},1,i_{3}},\dots,a_{i_{1},m,i_{3}})\in\rho$ and
$(a_{i_{1},i_{2},1},\dots,a_{i_{1},i_{2},m})\in\rho$. The (main)
diagonal of $(a_{i_{1},i_{2},i_{3}})$ is the $m$-tuple
$(a_{1,1,1},\dots,a_{m,m,m})$.

\emph{Remark:}
Let $A=\{1,\dots,k\}$. We mention that for an $n$-ary function $f:A^{n}\to A$ and a monoid $M\leq
A^{A}$ we have $f\in M^{*}$ if and only if $\Gamma_{M}\models
(a_{i_{1},\dots,i_{n}})$ where
$a_{i_{1},\dots,i_{n}}:=f({i_{1}},\dots,{i_{n}})$,
$i_{1},\dots,i_{n}\in\{1,\dots,k\}$. 

\begin{definitions}

For $\rho\subseteq A^{m}$ let $\rho^{\tra}$ denote the \New{transitive
closure} of $\rho$, i.e.,
$\rho^{\tra}=\bigcap\{\sigma\subseteq
A^{m}\mid \sigma \text{ is transitive and }\rho\subseteq\sigma\}$
is the least transitive relation containing $\rho$ (it is easy to
check that the intersection of transitive relations is again
transitive). Analogously, the \New{generalized quasiorder closure}
$\rho^{\gqu}$ is the least generalized quasiorder containing $\rho$.
The \New{reflexive closure} is naturally defined as
$\rho^{\rfl}:=\rho\cup\{(c,\dots,c)\in A^{m}\mid c\in A\}$. 
  
\end{definitions}

These
closures can be constructed (inductively) as follows.

\begin{proposition} \label{tra}
For $\rho\in \Rela[m](A)$ define 
$\partial(\rho):=\{(a_{11},\dots,a_{mm})\in A^{m}\mid \exists
(a_{ij})\in A^{m\times m}: \rho\models (a_{ij})\}$ and let $\rho^{(0)}:=\rho$,
$\rho^{(n+1)}:=\rho^{(n)}\cup\partial(\rho^{(n)})$ for $n\in\N$. Then we have
\begin{align*}
  \rho^{\tra}=\bigcup_{n\in\N}\rho^{(n)} \quad\text{ and }\quad
  \rho^{\gqu}=\rho^{\rfl\,\tra}.
\end{align*}
  
\end{proposition}

Remark: If $\rho$ is reflexive, then $\rho\subseteq\partial(\rho)$. For binary relations $\rho$ the operator $\partial$ is just
the relational product: $\partial(\rho)=\rho\circ\rho$.

\begin{lemma}\label{A1a}
  Let $\rho\in\gQuorda[m](A)$. Then for every $n$-dimensional
  $m\times\ldots\times m$-matrix
  $(a_{i_{1},\dots, i_{n}})_{i_{1},\dots,i_{n}\in\{1,\dots,m\}}$ we
  have
  \begin{align*}
    \rho\models (a_{i_{1},\dots, i_{n}})\implies 
    (a_{1,\dots,1},\dots,a_{m,\dots, m})\in\rho.
  \end{align*}
\end{lemma}

\begin{proof} For $n=2$ the
  condition follows from the definition of a generalized
  quasiorder. Thus we can assume $n\geq 3$. 
Let $M_{k}=(b^{k}_{i_{1},\dots,i_{n-k}})$ denote the $(n-k)$-dimensional
  $m\times\ldots\times m$-matrix with 
$b^{k}_{i_{1},\dots,i_{n-k}}:=a_{i_{1},\dots,i_{1},i_{2},\dots,i_{n-k}}$
 (the first $k$ coordinates are equal $i_{1}$).
. 
Thus $M_{0}=(a_{i_{1}\dots,i_{n}})$
  and $M_{n-1}=(b^{1}_{i})=(a_{i,\dots,i})_{i\in\{1,\dots,m\}}=
      (a_{1,\dots,1},\dots,a_{m,\dots, m})$.
We have to show $M_{n-1}\in\rho$ (formally $\rho\models M_{n-1}$). This
can be done by induction on $k$. By assumption we have $\rho\models
M_{k}$ for $k=0$. Assume $\rho\models M_{k}$ for some
$k\in\{0,1,\dots,n-2\}$. We are going to show $\rho\models M_{k+1}$ what
will finish the proof.

Let
  $i_{1},\dots,i_{n-k}\in\{1,\dots,m\}$. We fix
  $i_{3},\dots,i_{n-k}$ and consider the $m\times m$-matrix
  $M'_{k}:=(b^{k}_{i,j,i_{3},\dots,i_{n-k}})_{i,j\in\{1,\dots,m\}}$. Clearly, $\rho\models M_{k}$
  implies $\rho\models M'_{k}$. Therefore
  $(b^{k}_{1,1,i_{3},\dots,i_{n-k}},\dots,b^{k}_{m,m,i_{3},\dots,i_{n-k}})\in\rho$
  because $\rho$ is a generalized quasiorder. 
  Since
  $i_{3},\dots,i_{n-k}$ were chosen arbitrarily, this implies
  (together with $\rho\models M_{k}$) that we have
  $\rho\models M_{k+1}$ (note 
   $b^{k+1}_{i_{1},i_{3},\dots,i_{n-k}}=b^{k}_{i_{1},i_{1},i_{3},\dots,i_{n-k}}$). 
\end{proof}

One of the crucial properties of generalized quasiorders is that
preservation of a relation only depends on the translations, i.e., it
extends the property $\Xi$ (see~\eqref{MX1}) from (usual)
quasiorders to generalized quasiorders.

\begin{theorem}\label{A1b}
For $f\in \Op(A)$ and $\rho\in\gQuord(A)$ we have:
\begin{align*}
  f\preserves\rho \iff \trl{f}\preserves\rho.
\end{align*}  
Thus $\Xi(\rho)$ holds.
\end{theorem}

\begin{proof}
``$\Rightarrow"$: Since each $g\in\trl{f}$ is a composition of $f$ and
constants $c\in C$ and since constants preserve $\rho$ because of reflexivity,
we have $\trl{f}\subseteq\Sg{\{f\}\cup C}\preserves\rho$.
 ``$\Leftarrow$'': Let $\ar(f)=n$, $\ar(\rho)=m$, $\trl{f}\preserves \rho$ and let
 $r_{1},\dots,r_{n}\in\rho$. We are going to show
 $f(r_{1},\dots,r_{n})\in\rho$ what implies $f\preserves\rho$ and will
 finish the proof. 
Define $a_{i_{1},\dots,i_{n}}:=f(r_{1}(i_{1}),\dots,r_{n}(i_{n}))$. 
Then $a_{i_{1},\dots,[j],\dots,i_{n}}=f_{\bb,j}(r_{j})\in\rho$ for
$\bb=(r_{1}(i_{1}),\dots,r_{j-1}(i_{j-1}),r_{j+1}(i_{j+1}),\dots,r_{n}(i_{n}))$ (notation see
\eqref{M1a}) because $f_{\bb,j}\in\trl{f}\preserves\rho$, $j\in\{1,\dots,n\}$. Thus
$\rho\models(a_{i_{1},\dots,i_{n}})$ and we have
\begin{align*}
  f(r_{1},\dots,r_{n})&=(f(r_{1}(1),\dots,r_{n}(1)),\dots,f(r_{1}(m),\dots,r_{n}(m)))\\
  &=(a_{1,\dots,1},\dots,a_{m,\dots,m})\in\rho
\end{align*}
by \ref{A1a}, and we are done.
\end{proof}

\begin{corollary}\label{A1c}
Let $F\subseteq\Op(A)$ and $Q\subseteq\gQuord(A)$. Then
  \begin{enumerate}[\rm(i)]
  \item\label{A1ci} $\gQuord(A,F)=\gQuord(A,\trl{F})$ (cf.~Remark~\ref{M2X})
  \item\label{A1cii} $\Xi(Q)$ holds, i.e., $\Pol Q=(\End Q)^{*}$, in particular,
    $(\End Q)^{*}$ is a 
    clone and $\End Q$ is u-closed.
  \end{enumerate}
\end{corollary}

\begin{proof}
\eqref{A1ci} directly follows from \ref{A1b}. Concerning \eqref{A1cii},
we have 

$\Pol Q=\bigcap_{\rho\in Q}\Pol\rho
=_{\ref{A1b},\eqref{MX3}}\bigcap_{\rho\in Q}(\End\rho)^{*}
=_{\ref{M3B}}(\bigcap_{\rho\in Q}\End\rho)^{*}=(\End Q)^{*}$, i.e., $\Xi(Q)$.
\end{proof}

Now we characterize the u-closed monoids $M\leq A^{A}$
(i.e., $\uclose{M}=M$) by various properties.
The condition \ref{M7}\eqref{M7iii} will show that
the situation as in Example~\ref{M3A} is
characteristic for being \textbf{not} u-closed.

\begin{proposition}\label{M7}
 For a monoid $M\leq A^{A}$ the
  following are equivalent:
  \begin{enumerate}[\rm(i)]
  \item\label{M7i} $M$ is u-closed (equivalently, $M^{*}$ is a clone),
  \item\label{M7ii} $M^{*}=\Pol\Gamma_{M}$,
  \item\label{M7iii} $C\subseteq M$ and for every binary $f\in M^{*}$ we have $\Delta
    f\in M$,
  \item \label{M7iv}$\Gamma_{M}$ is a generalized quasiorder.
  \end{enumerate}

\end{proposition}

\begin{proof}  Each of the conditions \eqref{M7i}, \eqref{M7ii} and
\eqref{M7iv} implies $C\subseteq M$ (cf.\ \ref{M6} for \eqref{M7i},
\eqref{M7ii} and note that $\Gamma_{M}$ is reflexive if and only if
$C\subseteq M$). Thus we can assume $C\subseteq M$ in the following.

\eqref{M7ii}$\implies$\eqref{M7i}$\implies$\eqref{M7iii}
  is clear (each set of the form $\Pol Q$ is a clone, and any clone is
  closed under $\Delta$).

\eqref{M7i}$\implies$\eqref{M7ii}: $M$ is just the unary part $F^{(1)}$
  of the clone $F:=M^{*}$.
  It is well-known (cf., e.g., \cite[Proposition 3.1]{Sze86}) that $\Pol\Gamma_{M}$ is the
  largest clone $F$ with unary part $F^{(1)}=M$, thus $M^{*}=F\subseteq
  \Pol\Gamma_{M}$.

Conversely, let $f\in\Pol\Gamma_{M}$, i.e.,
  $f\preserves\Gamma_{M}$. Remember that the elements of $\Gamma_{M}$
  are of the form $\br_{g}$ for some $g\in M$ (notation
  see~\ref{M5B}). Thus $f\preserves\Gamma_{M}$ means
  $f(\br_{g_{1}},\dots,\br_{g_{n}})\in\Gamma_{M}$ whenever
  $g_{1},\dots,g_{n}\in M$. Since
  $f(\br_{g_{1}},\dots,\br_{g_{n}})=\br_{f[g_{1},\dots,g_{n}]}$,  this equivalently can be
  expressed by the condition that the composition
  $f[g_{1},\dots,g_{n}]$ belongs to $M$ whenever 
  $g_{1},\dots,g_{n}\in M$.
Consequently, any translation $g:=f_{\mathbf{a},i}$ derived from $f$ 
  (w.l.o.g.\ we take $i=1$), say
  $g(x):=f(x,a_{2},\dots,a_{n})$ for some $a_{2},\dots,a_{n}\in A$, must belong to $M$, since
  $g=f[\id_{A},\bc_{a_{2}},\dots,\bc_{a_{n}}]$ and $M$ contains the identity
  $\id_{A}$ and the constant functions. Thus $\trl{f}\subseteq M$, hence
  $f\in M^{*}$, and we get $\Pol\Gamma_{M}\subseteq M^{*}$.

\eqref{M7iii}$\implies$\eqref{M7i}: Assume \eqref{M7iii} and assume on the
contrary that $M^{*}$ is not a clone. We lead this to a contradiction.
 Since $M^{*}$ is a preclone
by~\ref{M6}, $M^{*}$ cannot be closed under $\Delta$ and there must
exist a function $f\in M^{*}$, say $n$-ary, such that $h:=\Delta
f\notin M^{*}$ (clearly $n\geq 3$, otherwise we have a contradiction
to \eqref{M7iii}). Thus some translation $g:=h_{\mathbf{a},i}$ derived from $h$
cannot belong to $M$. If $i\not= 1$, then
$g(x)=h(c_{1},\dots,c_{i-1},x,c_{i+1}\dots,c_{n-1})=f(c_{1},c_{1},c_{i-1},x,c_{i+1}\dots,c_{n-1})$
would belong to $M$ since $f\in M^{*}$. Therefore $i=1$ and
$g(x)=h(x,c_{2},\dots,c_{n-1})=f(x,x,c_{2},\dots,c_{n-1})$ does not
belong to $M$. Consider the binary function
$f'(x_{1},x_{2}):=f(x_{1},x_{2},c_{2},\dots,c_{n-1})$. We have $f'\in
M^{*}$ (since $f\in M^{*}$) and $\Delta f'\notin M$ (since $g=\Delta
f'$ by definition), in contradiction to~\eqref{M7iii}.

\eqref{M7iii}$\iff$\eqref{M7iv}: Let $A=\{1,\dots,k\}$. There is a
bijection between binary operations $f:A^{2}\to A$ and $(k\times k)$-matrices
$(a_{ij})$ via $a_{ij}=f(i,j)$ for $i,j\in\{1,\dots,k\}$. Note that
rows and colums of 
$(a_{ij})$ are just the function tables $(f(i,1),\dots,f(i,k))$ and
$(f(1,j),\dots,f(k,j))$ of the translations
$f(i,x)$ and $f(x,j)$. Therefore $f\in M^{*}$ (i.e.,
$\trl{f}\subseteq M$ by definition) is equivalent to the property that
all rows and colums of $(a_{ij})$ belong to $\Gamma_{M}$ (since the
colums of $\Gamma_{M}$ are just the function tables of the unary
functions in $M$), i.e., $\Gamma_{M}\models (a_{ij})$. Further, $\Delta f\in M$ is equivalent to the
property that the diagonal $(a_{11},\dots,a_{kk})$ of $(a_{ij})$
belongs to $\Gamma_{M}$. Thus condition $\eqref{M7iii}$ is equivalent
to the reflexivity (because $C\subseteq M$) and transitivity of $\Gamma_{M}$ (according to \ref{A1}), and
therefore to $\Gamma_{M}$ being a generalized quasiorder.
\end{proof}

The following corollary is a simple tool to construct functions in
the u-closure of a monoid. 

\begin{corollary}\label{M8} Let $A=\{1,\dots,k\}$ and $M\leq
  A^{A}$. If, for a binary operation $h:A^{2}\to A$, we have 
$h\in(\uclose{M})^{*}$, in particular if $h\in M^{*}$, then $\Delta h\in
  \uclose{M}$. 
\end{corollary}

\begin{proof}
  The statement is just \ref{M7}\eqref{M7iii} for the u-closed monoid $\uclose{M}$.
\end{proof}

We mention further that $h\in(\uclose{M})^{*}$ is equivalent to
$\Gamma_{\uclose{M}}\models V$ for the matrix $V:=(h(i,j))_{i,j\in A}$.



\section{The Galois connection $\boldsymbol{\End-\gQuord}$}
\label{secB} 

\begin{EndgQuord}\label{A2}
The preservation property $\preserves$ induces a Galois connection 
between unary mappings and generalized quasiorders given by the
operators
\begin{align*}
  \End Q&:=\{h\in A^{A}\mid \forall \rho\in Q: h\preserves\rho\} \text{
  (endomorphisms) and}\\
  \gQuord M&:=\{\rho\in\gQuord(A)\mid \forall h\in M:
             h\preserves\rho\} \text{ (generalized quasiorders)}
  \end{align*}
 for
$M\subseteq A^{A}$ and $Q\subseteq\gQuord(A)$. The corresponding
Galois closures are $\End\gQuord M$ and $\gQuord\End Q$.
\end{EndgQuord}

\normalsize

Now we can show one of our main results, namely that the
u-closed monoids are just the Galois closures with respect to the Galois
connection $\End - \gQuord$. As a consequence (as shown in~\ref{A3cor} and~\ref{A3Xi}) we can answer the
questions raised in the Introduction.

\begin{theorem}\label{A3} Let $M\subseteq A^{A}$. Then we have:
  \begin{align*}
    \uclose{M}=\End\gQuord M.
  \end{align*}

\end{theorem}

\begin{proof} 
At first we observe $M\subseteq\End\gQuord M$
(this holds for every Galois connection),
$M\subseteq\uclose{M}=\End \Gamma_{\uclose{M}}$, in particular
$M\preserves \Gamma_{\uclose{M}}$, and by
\ref{M7}\eqref{M7iv} we know $\Gamma_{\uclose{M}}\in\gQuord(A)$. Thus
$\Gamma_{\uclose{M}}\in \gQuord M$. Consequently we get
$\uclose{M}\subseteq\uclose{\End\gQuord M}=_{\ref{A1c}\eqref{A1cii}}\End\gQuord M\subseteq\End
\Gamma_{\uclose{M}}=\uclose{M}$, and we are done.
 \end{proof}

In addition to the characterization in \ref{M7} we give some further
consequenses of Theorem~\ref{A3}, characterizing $M^{*}$ (\ref{A3cor}\eqref{A3a}) and u-closed
monoids $M$ (\ref{A3cor}\eqref{A3b}). Since every monoid can be given as endomorphism monoid of
invariant relations, $M=\End Q$, we also look for the characterization
of those $Q$ with u-closed endomorphism monoid (\ref{A3cor}\eqref{A3c}):
\begin{corollary}\label{A3cor}
  \begin{enumerate}[\rm(a)]
  \item\label{A3a} $(\uclose{M})^{*}=\Pol\gQuord M$ for $M\subseteq
    A^{A}$.

  \item\label{A3b} 
   The following are equivalent for $M\leq A^{A}$:
    \begin{itemize}
    \item[\rm(i)] $M$ is u-closed,
   \qquad {\rm(i)$'$} $M^{*}$ is a clone,
   \qquad {\rm(i)$''$} $\Gamma_{M}\in\gQuord(A)$,
    \item[\rm(ii)] $M=\End Q$ for some $Q\subseteq\gQuord(A)$,
    \item[\rm(iii)] $M^{*}=\Pol Q$ for some $Q\subseteq\gQuord(A)$,
    \end{itemize}
where the same $Q$ can be taken in {\rm(ii)} and {\rm(iii)}.

  \item\label{A3c}
   The following are equivalent for $Q\subseteq\Rel(A)$:
    \begin{itemize}
    \item[\rm(i)] $\End Q$ is u-closed,
    \quad {\rm(i)$'$ }  $(\End Q)^{*}$ is a clone,
    \quad {\rm(i)$''$ }  $\Gamma_{\End Q}\in\gQuord(A)$,
    \item[\rm(ii)]  $\exists
      Q'\subseteq\gQuord(A): \End Q=\End Q'$, 
    \item[\rm(ii)$'$] $\exists Q'\subseteq\gQuord(A):
      [Q]_{\exists,\land,\lor,=}=[Q']_{\exists,\land,\lor,=}$\\ 
   (closure under positive formulas)
    \item[\rm(iii)] $\exists Q'\subseteq\gQuord(A): (\End Q)^{*}=\Pol Q'$,
    \end{itemize}
where the same $Q'$ can be taken in {\rm(ii)} and {\rm(iii)}.
Instead of `` $\exists Q'\subseteq\gQuord(A)$'' one can take
`` $\exists\rho\in\gQuord(A)$'' and $Q'=\{\rho\}$.
 
\end{enumerate}
\end{corollary}

\begin{proof}
\eqref{A3a}: Let $Q:=\gQuord M$. Then $\Xi(Q)$ by \ref{A1b}, i.e.,
$\Pol Q=(\End Q)^{*}$ (cf.~\eqref{MX1}). Thus $\Pol Q=(\End\gQuord
M)^{*}=(\uclose{M})^{*}$ by \ref{A3}.

\eqref{A3b}: For (i)$\iff$(i)$'$$\iff$(i)$''$ see
\ref{M3D}\eqref{M3Diii} and \ref{M7}\eqref{M7iv}.

(i)$\implies$(ii): Take $Q:=\gQuord M$. If $M$ is u-closed, then
$M=\uclose{M}=_{\eqref{A3}}\End Q$.

(ii)$\implies$(iii): $(\End Q)^{*}=\Pol Q$ directly follows from
\ref{A1c}\eqref{A1cii}.

(iii)$\implies$(i)$'$ is obvious, because $M^{*}=\Pol Q$ is a clone.

\eqref{A3c} is just \eqref{A3b} for $M=\End Q$. (ii)$\iff$(ii)$'$
follows from the properties of the Galois connection $\End-\Inv$
(in particular $[Q]_{\exists,\land,\lor,=}=\Inv\End Q$,
cf.~\ref{PolInv}).
Further note, that $Q'=\{\Gamma_{\End Q}\}$ also will do the job (instead  of
arbitrary $Q'$) since $\End Q=\End\Gamma_{\End Q}$.  
\end{proof}

Now we are also able to answer the question which (sets of) relations
satisfy the property $\Xi$ (cf.~\ref{MX}):

\begin{proposition}\label{A3Xi}  The following are equivalent for $Q\subseteq\Rel(A)$:
  \begin{itemize}
    \item[\rm(i)] $\Xi(Q)$ holds, i.e., $\Pol Q=(\End Q)^{*}$,

    \item[\rm(ii)] $\exists Q'\subseteq\gQuord(A): \Pol Q=\Pol Q'$,

    \item[\rm(ii)$'$] $\exists Q'\subseteq\gQuord(A):
      [Q]_{\exists,\land,=}=[Q']_{\exists,\land,=}$\\ 
   (closure under primitive positive formulas),

    \item[\rm(ii)$''$] $[Q]_{\exists,\land,=}=[[Q]_{\exists,\land,=}\cap\gQuord(A)]_{\exists,\land,=}$.
    \end{itemize}

\end{proposition}

\begin{proof}
  (i)$\implies$(ii): Assume $\Pol Q=(\End Q)^{*}$ and let $M:=\End Q$
  and $Q':=\gQuord M$. $M$ is u-closed
  (since $M^{*}$ is a clone), therefore
  $\Pol Q=M^{*}=(\uclose{M})^{*}=_{\ref{A3cor}\eqref{A3a}}\Pol Q'$.

(ii)$\implies$(i): Assume $\Pol Q=\Pol Q'$ ($Q'\subseteq \gQuord(A)$).
Then $\End Q=\End Q'$ and we have $\Pol Q=\Pol Q'=_{\ref{A1c}\eqref{A1cii}}(\End
Q')^{*}=(\End Q)^{*}$, consequently $\Xi(Q)$ by Definition~\ref{MX}.

(ii)$\iff$(ii)$'$
follows from the properties of the Galois connection $\Pol-\Inv$
(in particular $[Q]_{\exists,\land,=}=\Inv\Pol Q$, cf.~\ref{PolInv}).
(ii)$'$$\iff$(ii)$''$ is obvious.
\end{proof}

\begin{remark}\label{A3New} We know from \ref{A1b} that
  $\rho\in\gQuord(A)$ implies 
  $\Xi(\rho)$.  The converse is not true:
  $\Xi(\rho)$ does not imply  $\rho\in\gQuord(A)$ in general!  
A
counterexample is the binary relation $\rho=\{(i,j)\mid 1\leq i,j\leq
n,\; j\leq i+1,\; j\neq i-1\}$
in \cite[Example~3.5]{LaeP84} on an at least
$5$-element set $A=\{1,\dots,n\}$. This relation is
strongly $C$-rigid (what means $\Pol\rho=\Sg{\{\id_{A}\}\cup C}$)
and reflexive, but not transitive, i.e.,
$\rho\notin\gQuord(A)$. Nevertheless  $\Xi(\rho)$ holds. To see this
we have to show $\Pol\rho=M^{*}$, where $M:=\End\rho$, i.e., $M=\{\id_{A}\}\cup
C=T$.
$M^{*}$ is u-closed (what we shall prove in~\ref{Btriv}), thus  $M^{*}=\Pol\Gamma_{M}$ by
\ref{M7}\eqref{M7ii}. 
By \cite[Proposition~2.2.]{LaeP84}, for a clone $F$, if its unary part
$F^{(1)}$ equals
$\{\id_{A}\}\cup C$, then $F=\Sg{\{\id_{A}\}\cup C}$. Consequently, for
$F=M^{*}$ we have $F^{(1)}=M=\{\id_{A}\}\cup C$ and therefore we get $M^{*}=\Sg{\{\id_{A}\}\cup C}=\Pol\rho$.

For $n=5$ we get the relation $\rho$ shown in 
Figure~\ref{FigB}
(this is a so-called tournament). 

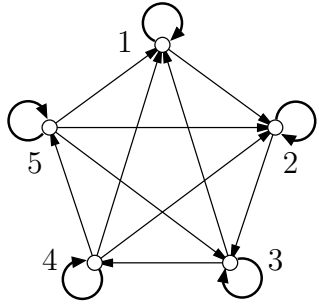
\begin{figure}
  \begin{center}
\begin{picture}(0,0)%
\includegraphics{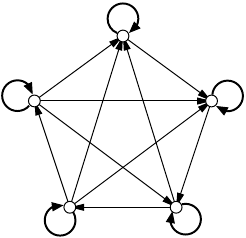}%
\end{picture}%
\setlength{\unitlength}{4144sp}%
\begingroup\makeatletter\ifx\SetFigFont\undefined%
\gdef\SetFigFont#1#2#3#4#5{%
  \reset@font\fontsize{#1}{#2pt}%
  \fontfamily{#3}\fontseries{#4}\fontshape{#5}%
  \selectfont}%
\fi\endgroup%
\begin{picture}(1866,1800)(2484,-1327)
\put(3151,164){\makebox(0,0)[lb]{\smash{{\SetFigFont{12}{14.4}{\rmdefault}{\mddefault}{\updefault}{\color[rgb]{0,0,0}$1$}%
}}}}
\put(4141,-556){\makebox(0,0)[lb]{\smash{{\SetFigFont{12}{14.4}{\rmdefault}{\mddefault}{\updefault}{\color[rgb]{0,0,0}$2$}%
}}}}
\put(4051,-1141){\makebox(0,0)[lb]{\smash{{\SetFigFont{12}{14.4}{\rmdefault}{\mddefault}{\updefault}{\color[rgb]{0,0,0}$3$}%
}}}}
\put(2701,-1141){\makebox(0,0)[lb]{\smash{{\SetFigFont{12}{14.4}{\rmdefault}{\mddefault}{\updefault}{\color[rgb]{0,0,0}$4$}%
}}}}
\put(2611,-556){\makebox(0,0)[lb]{\smash{{\SetFigFont{12}{14.4}{\rmdefault}{\mddefault}{\updefault}{\color[rgb]{0,0,0}$5$}%
}}}}
\end{picture}%
  \caption{A reflexive but not transitive relation $\rho$ with $\Xi(\rho)$}\label{FigB}
\end{center}
\end{figure}

Nevertheless, by \ref{A3Xi}, $\rho$ must be  ``constructively
equivalent''  to some 
$Q'\subseteq\gQuord(A)$,
  i.e., $[\rho]_{\exists,\land,=}=[Q']_{\exists,\land,=}$.
  In this concrete case we can take $Q'=\{\Gamma_{M}\}$, i.e., we
   have $[\rho]_{\exists,\land,=}=[\Gamma_{M}]_{\exists,\land,=}$, since
  $\Pol\rho=\Pol\Gamma_{M}$. 
\end{remark}

Before we investigate the u-closure for concrete monoids we show how
this closure behaves under taking products and substructures.
 For this we need some notation.

\begin{definition}\label{A4}
Let $g_{i}\in A_{i}^{A_{i}}$ ($i\in\{1,2\}$) and $A=A_{1}\times
A_{2}$. Then $g:=g_{1}\otimes g_{2}$ denotes the unary operation $g\in
A^{A}$ defined componentwise by
$g(a_{1},a_{2}):=(g_{1}a_{1},g_{2}a_{2})$. For $M_{i}\subseteq
A_{1}^{A_{i}}$  we put $M_{1}\otimes M_{2}:=\{g_{1}\otimes g_{2}\mid
 g_{1}\in M_{1}\text{ and }g_{2}\in M_{2}\}$.

Further, for $\rho_{i}\in\Rela[m](A_{i})$ and  $Q_{i}\subseteq
\Rel(A_{i})$, $i\in\{1,2\}$, let
\begin{align*}
  \rho_{1}\otimes\rho_{2}&:=\{((a_{1},b_{1}),\dots,(a_{m},b_{m}))\mid
(a_{1},\dots,a_{m})\in\rho_{1}\text{ and
}(b_{1},\dots,b_{m})\in\rho_{2}\},\\
Q_{1}\otimes Q_{2}&:=\{\rho_{1}\otimes \rho_{2}\mid
 \rho_{1}\in Q^{(m)}_{1}\text{ and }\rho_{2}\in Q^{(m)}_{2},\, m\in\N_{+}\}.
\end{align*}
Remark: For monoids $M_{1}, M_{2}$, the product $M_{1}\otimes M_{2}$ is
isomorphic (as monoid) to the direct product $M_{1}\times M_{2}$.
\end{definition}

\begin{proposition}\label{A5}
Let $\id_{A_{i}}\in M_{i}\subseteq A_{i}^{A_{i}}$, $i\in\{1,2\}$ and
$A=A_{1}\times A_{2}$. Then we have
\begin{enumerate}[\rm(a)]
\item \label{A5a} $\gQuord_{A}(M_{1}\otimes
  M_{2})=(\gQuord_{A_{1}}M_{1})\otimes (\gQuord_{A_{2}}M_{2})$.

\item \label{A5b} $\uclose{M_{1}\otimes M_{2}}=\uclose{M_{1}}\otimes \uclose{M_{2}}$.
\end{enumerate}
\end{proposition}

\begin{proof}
 \eqref{A5a}: According to \cite[2.3.7]{PoeK79} and because the identity map
  belongs to $M_{i}$, for the invariant relations we have
  $\Inv_{A}(M_{1}\otimes M_{2})=(\Inv_{A_{1}}M_{1})\otimes
  (\Inv_{A_{2}}M_{2})$. Thus, in order to prove \eqref{A5a}, it only remains to show that 
  \begin{align*}
    \rho_{1}\otimes\rho_{2}\in\gQuord(A) \iff
    \rho_{1}\in\gQuord(A_{1}) \text{ and }\rho_{2}\in\gQuord(A_{2})
  \end{align*}
for $\rho_{1}\in\Rela[m](A_{1})$ and $\rho_{2}\in\Rela[m](A_{2})$. But this
follows from (notation see~\ref{A1})
\begin{align*}
  \rho_{1}\models (a_{ij})\text{ and } \rho_{2}\models (b_{ij})&\iff
  (\rho_{1}\otimes \rho_{2})\models ((a_{ij},b_{ij})),\text{ and}\\
 (a_{11},\dots,a_{mm})\in\rho_{1}\text{ and
  }&(b_{11},\dots,b_{mm})\in\rho_{2}\\&\iff 
((a_{11},b_{11}),\dots,(a_{mm},b_{mm}))\in\rho_{1}\otimes\rho_{2},
\end{align*}
what is clear from the definitions~\ref{A4}.

\eqref{A5b}: Since the trivial equivalence relations $\Delta_{A_{i}}$
and $\nabla_{A_{i}}$ belong to $\gQuord_{A_{i}}M_{i}$ ($i\in\{1,2\}$),
we can apply 
$(\Pol_{A_{1}}Q_{1})\otimes (\Pol_{A_{2}}Q_{2})=\Pol_{A}(Q_{1}\otimes Q_{2})$ 
from \cite[Satz 2.3.7(vi) and \textit{\"Ub} 2.4, p.73]{PoeK79}
(restricting to unary 
mappings, i.e., taking $\End$ instead of $\Pol$ and $Q_{i}=\gQuord
M_{i}$) in order to get the 
second equality in the following conclusions:
\begin{align*}
  \uclose{M_{1}}\otimes \uclose{M_{2}}
      &\stackrel{\ref{A3}}{=}(\End\gQuord M_{1})\otimes(\End\gQuord M_{2})\\ 
      &=\End((\gQuord M_{1})\otimes (\gQuord M_{2}))\\
      &\stackrel{\eqref{A5a}}{=}\End\gQuord(M_{1}\otimes M_{2})\\
      &\stackrel{\ref{A3}}{=}\uclose{M_{1}\otimes M_{2}}.
\end{align*}
\end{proof}

\begin{proposition}\label{A6} Let $M\subseteq A^{A}$ and $B\in \Inv M$
  for some $\emptyset\neq B\subset A$. Then
  \begin{align*}
\gQuord_{B}(M\restriction_{B})=(\gQuord_{A}M)\restriction_{B}.
  \end{align*}
\end{proposition}

\begin{proof}
According to \cite[2.3.4]{PoeK79} we have
$\Inv_{B}(M\restriction_{B})=(\Inv_{A}M)\restriction_{B}$. Thus it remains to show that the generalized
quasiorders (which are special invariant relations) correspond to each
other, more precisely, we have to prove that
\begin{align*}
  \gQuord(B)=(\gQuord(A))\restriction_{B}.
\end{align*}
``$\subseteq$'': Let $\sigma\in\gQuorda[m](B)$ and 
$\rho:=\sigma\cup\{(a,\dots,a)\in A^{m}\mid a\in A\setminus B\}$.
Then $\sigma=\rho\restriction_{B}$. Moreover, $\rho$ is reflexive by
construction. To show transitivity, let $\rho\models (a_{ij})\in
A^{m\times m}$. If $(a_{ij})\in B^{m\times m}$, then
$\sigma\models(a_{ij})$ and we get
$(a_{11},\dots,a_{mm})\in\sigma\subseteq\rho$ (since $\sigma$ is
transitive).
If some row or column of $(a_{ij})$ contains an element $a\in
A\setminus B$, then by definition of $\rho$ this row or column must be
$(a,\dots,a)$. Thus $a_{ij}=a$ for all $i,j$, and the
diagonal obviously belongs to $\rho$. Thus $\rho$ is transitive, i.e.,
$\rho\in\gQuord(A)$.

``$\supseteq$'': If $\rho\in\gQuord(A)$ then
$\sigma:=\rho\restriction_{B}\in\Rel(B)$ is obviously reflexive (on
$B$) and also transitive (since each matrix $(b_{ij})\in B^{m\times m}$ can
be considered as a matrix in $A^{m\times m}$). Thus $\sigma\in\gQuord(B)$.
\end{proof}

\begin{remark}\label{R1}
  We do not consider here the other side of the Galois connection, i.e.,
  the Galois closures of the form $\gQuord \End Q$ for $Q\subseteq\gQuord(A)$. In general,
they are not relational clones (contrary to the Galois connection
$\End -\Inv$). In particular, $\Quord(A)$ is not a
relational clone. It contains all diagonal relations and is closed
under several relational clone operations, but, e.g., not under $\pr$
(i.e., deleting of coordinates). 
For example, the relation
$\rho:=\{(0,0,0),(1,1,1),(2,2,2),(2,0,1),(1,1,2)\}$ on $A=\{0,1,2\}$
is a generalized quasiorder (this is easy to check), but
$\pr_{2,3}(\rho)=\{(x,y)\mid \exists
a:(a,x,y)\in\rho\}=\{(0,0),(1,1),(2,2),(0,1),(1,2)\}$ is not (because
it is not transitive).
\end{remark}

\section{Minimal u-closed monoids}\label{secC}

In this section we investigate some special monoids and their u-closure. For a unary
function $f\in A^{A}$ let $M_{f}:=\Sg{f}\cup C$. This is the
least monoid containing $f$ and all constants.
What can be said about the u-closure of such monoids $M_{f}$? 

In the following we have to deal much with the relation $\Gamma_{M}$
for a monoid $M=M_{f}$ and with the situation that $\Gamma_{M}\models
V$ for some $k\times k$-matrix $V=(v_{ij})$, $k:=|A|$. Therefore it is
convenient to identify a $g\in M$ with the vector
$\br_{g}=(ga_{1},\dots,ga_{k})$ (cf.~\ref{M5B}, here we assume
$A=\{a_{1},\dots,a_{k}\}$ where $A$ is implicitly ordered by the
indices of $a_{i}$). Thus we can say that a row or column $\br$
of $V$ equals some ``vector'' ($k$-tuple) $g\in M$ and write
$\br=g$ meaning $\br=(ga_{1},\dots,ga_{k})$. This will
be used very often in the proofs (in great detail in the proof of
\ref{Btriv}). Furthermore, let $\bv_{i,*}:=(v_{i1},\dots,v_{ik})$ and
$\bv_{*,i}:=(v_{1i},\dots,v_{ki})$ denote the $i$-th row and the $i$-th
column of $V=(v_{ij})$, respectively ($i\in\{1,\dots,k\}$).
Note that $\Gamma_{M_{f}}$ is reflexive since $M_{f}$ contains all
constants.

For the
trivial monoid $T:=M_{\id_{A}}=\{\id_{A}\}\cup C$ we have:

\begin{proposition}\label{Btriv}
  The monoid $T=\{\id_{A}\}\cup C$ is u-closed.
\end{proposition}
\begin{proof} Let $A=\{a_{1},\dots,a_{k}\}$.
  We show that $\Gamma_{T}$ is a generalized quasiorder (then
  we are done due to \ref{M7}\eqref{M7iv}). $\Gamma_{T}$
  is reflexive, thus it remains to show that
  $\Gamma_{T}$ is transitive. Let $V=(v_{ij})_{i,j\in \{1,\dots,k\}}$ be a
  $k\times k$-matrix such that $\Gamma_{T}\models V$, i.e., each row
  and each column  is one of the ``vectors'' $g\in T$, namely
 $\id_{A}=(a_{1},\dots,a_{k})$ or one of the constants $\bc_{1}=(a_{1},\dots,a_{1})$,\dots, $\bc_{k}=(a_{k},\dots,a_{k})$
  ($\bc_{i}$ denotes the constant mapping $\bc_{i}(x)=a_{i}$). If
  $v_{jj}=a_{i}$ for some $i\neq j$, then $\Gamma_{T}\models V$ can
  hold only if all rows and columns are equal to the constant $\bc_{i}$
  (since $\bc_{i}$ is the only vector where $a_{i}$ is on the $j$-th place),
  in particular, the main diagonal of $V$ also equals $\bc_{i}$ and
  therefore belongs to $\Gamma_{T}$. It remains the case
  $v_{ii}=a_{i}$ for all $i\in\{1,\dots,k\}$. Then the diagonal of $V$
  is $\id_{A}$, also belonging to $\Gamma_{T}$. Consequently,
  $\Gamma_{T}$ is transitive.
\end{proof}

For $|A|=2$ there exist only two monoids containing all constants,
namely $T$ and $A^{A}$, both are u-closed (the first by~\ref{Btriv},
the second trivially). Therefore, in the
following, we always can assume $|A|\geq 3$.

We are going to characterize the \New{minimal u-closed monoids}, i.e.,
u-closed monoids $M\leq
A^{A}$ which properly contain no other u-closed monoid except the
\New{trivial monoid}  $T=\{\id_{A}\}\cup C$. Such minimal u-closed
monoids must be generated by 
a single function, i.e., they must be of the form $\uclose{M_{f}}$ for
some unary $f$, moreover, $M_{f}$ can be assumed to be
\New{$C$-minimal}, i.e., minimal among
all monoids properly containing $T$ (otherwise $M_{f'}< M_{f}$ would imply
$\uclose{M_{f'}}\leq\uclose{M_{f}}$ and $\uclose{M_{f}}$ could be
canceled in the list of minimal u-closed monoids). 

It is well-known which unary functions $f$ generate a $C$-minimal monoid
$M_{f}\leq A^{A}$ (it follows, e.g., from \cite[4.1.4]{PoeK79}),
namely if and only if
 $f\in A^{A}$ is a \textsl{nontrivial} (i.e., $f\notin T$) function satisfying 
one of the following conditions:
\begin{enumerate}[\rm(i)]
\item $f^{2}=f$,
\item $f^{2}$ is constant,
\item $f$ is a permutation, such that $f^{p}=\id_{A}$ for some prime
  number $p$.
\end{enumerate}
As shown in \cite[Theorem~3.1]{JakPR2016}, among these functions are those for
which the quasiorder lattice $\Quord f$ is ma\-ximal among all
quasiorder lattices (on $A$), 
equivalently, for which $\End\Quord f$ is minimal (among all
endomorphism monoids of quasiorders). These functions are of
so-called type \rmI, \rmII{} and \rmIII, defined as follows:
\begin{enumerate}[\rm(I)]
\item $f^{2}=f$,
\item $f^{2}$ is constant, say $v$, and $|\{x\in A\mid fx=v\}|\geq 3$,
\item $f$ is a permutation with at least two cycles of length $p$, such that $f^{p}=\id_{A}$ for some prime number $p$.
\end{enumerate}
Note that $\Sg{f}=\{\id_{A},f\}$ for $f$ of type
  \rmI{} and \rmII, and $\Sg{f}=\{\id_{A},f,f^{2},\dots,f^{p-1}\}$ is
  a cyclic group of prime order for $f$ of type \rmIII.

Surprisingly it turns out (see Theorem~\ref{B0min}) that for each candidate $M_{f}$ with $f$
satisfying (i)--(iii), the u-closure $\uclose{M_{f}}$ is either not a
minimal u-closed monoid or $M_{f}$ itself is already u-closed. Thus
the minimal u-closed monoids coincide with the u-closed $C$-minimal
monoids.
 We start with
the functions of type \rmI, \rmII{} and \rmIII.

\begin{proposition}\label{B0}
  Let $f$ be a function of type \rmI, \rmII{} or \rmIII. Then $M_{f}$
  is a minimal u-closed monoid, in particular
  $M_{f}=\uclose{M_{f}}$. Moreover we have
  $\End\gQuord M_{f}=\End\Quord M_{f}$.
\end{proposition}

\begin{proof}
  Clearly, $\Sg{f}\cup C=M_{f}\subseteq\End\gQuord
  M_{f}\subseteq\End\Quord M_{f}$. But we have  $\End\Quord
  M_{f}=\Sg{f}\cup C$ as it was explicitly stated in
  \cite[Theorem~2.1(B)]{JakPR2023} 
  (but it already follows from the results in \cite{Jak1982},
  \cite{Jak1983} and also from \cite[Prop.~4.8]{JakPR2018}). Thus we have
  equality instead of the above inclusions and  $M_{f}$ is
  u-closed (by Theorem~\ref{A3}). Since $M_{f}$ has no proper submonoids except $T$ because
  $f$ satisfies one of the above conditions (i)--(iii), it is a
  minimal u-closed monoid. 
\end{proof}

\begin{theorem}\label{B0min} Let $3\leq|A|<\infty$.
The minimal u-closed monoids $M\le A^{A}$ are exactly
  those of the form $M_{f}=\Sg{f}\cup C$ where $f\in A^A$ is nontrivial and satisfies
  \begin{itemize}
\item[\rm(I)] $f^2=f$, or
\item[\rm(II$'$)] $f^2$ is a constant and $|A|\geq 4$,
  or
\item[\rm(III$'$)] $f^p=\id_A$ for some prime $p$ such that $f$ has at
  least two fixed points or $f$ is of type~\rmIII.
  \end{itemize}
In particular, each minimal u-closed monoid is $C$-minimal, too.
\end{theorem}
\begin{proof}
\textsc{Part 1:} At first we show that $M_{f}$ is u-closed for  all functions of type
\rmI, and of the new type \rmII$'$ or \rmIII$'$. Because of Proposition~\ref{B0}, it
remains to check only those functions which are of type \rmII$'$ or
\rmIII$'$, but not of type \rmII{} or \rmIII{}, respectively.

\underline{Case 1:} \textit{$f$ is of type \rmII$'$ but not of type \rmII}, i.e., $f^{2}$
is constant, denoted by $1$, $|\{x\in A\mid fx=1\}|=2$ and $|A|\geq 4$.

For simplicity we denote the elements of $A$ by numbers,
$A=\{1,2,\dots,k\}$, where $f1=1$ and $f2=1$ (otherwise $fx=2$),
$k\geq 4$. 
Thus $f$ has the form as given in Figure~\ref{Fig2}(a).
Observe that $M_{f}=\{\id_{A},f,\bc_{1},\dots,\bc_{k}\}$ ($\bc_{i}$
denotes the constant function $i$).

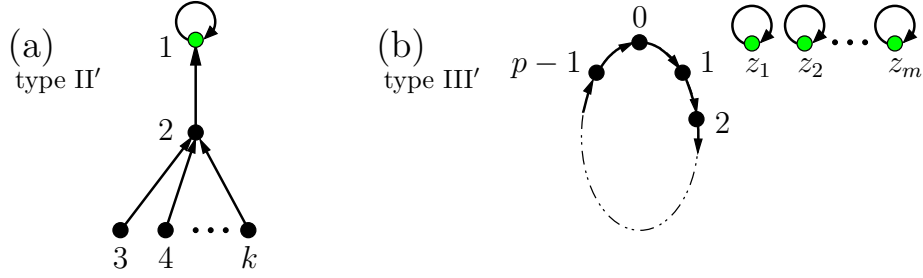
\begin{figure} \begin{center}
\begin{picture}(0,0)%
\includegraphics{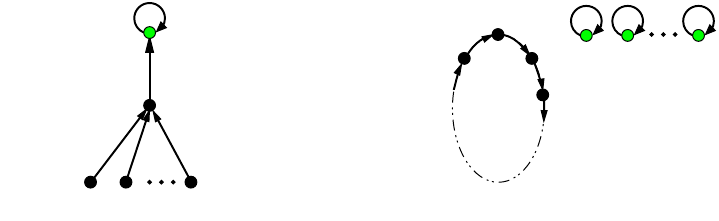}%
\end{picture}%
\setlength{\unitlength}{4144sp}%
\begingroup\makeatletter\ifx\SetFigFont\undefined%
\gdef\SetFigFont#1#2#3#4#5{%
  \reset@font\fontsize{#1}{#2pt}%
  \fontfamily{#3}\fontseries{#4}\fontshape{#5}%
  \selectfont}%
\fi\endgroup%
\begin{picture}(5466,1675)(436,-3010)
\put(4861,-1771){\makebox(0,0)[lb]{\smash{{\SetFigFont{12}{14.4}{\rmdefault}{\mddefault}{\updefault}{\color[rgb]{0,0,0}$z_{1}$}%
}}}}
\put(5176,-1771){\makebox(0,0)[lb]{\smash{{\SetFigFont{12}{14.4}{\rmdefault}{\mddefault}{\updefault}{\color[rgb]{0,0,0}$z_{2}$}%
}}}}
\put(5716,-1771){\makebox(0,0)[lb]{\smash{{\SetFigFont{12}{14.4}{\rmdefault}{\mddefault}{\updefault}{\color[rgb]{0,0,0}$z_{m}$}%
}}}}
\put(4186,-1501){\makebox(0,0)[lb]{\smash{{\SetFigFont{12}{14.4}{\rmdefault}{\mddefault}{\updefault}{\color[rgb]{0,0,0}$0$}%
}}}}
\put(3466,-1771){\makebox(0,0)[lb]{\smash{{\SetFigFont{12}{14.4}{\rmdefault}{\mddefault}{\updefault}{\color[rgb]{0,0,0}$p-1$}%
}}}}
\put(4591,-1771){\makebox(0,0)[lb]{\smash{{\SetFigFont{12}{14.4}{\rmdefault}{\mddefault}{\updefault}{\color[rgb]{0,0,0}$1$}%
}}}}
\put(1351,-1681){\makebox(0,0)[lb]{\smash{{\SetFigFont{12}{14.4}{\rmdefault}{\mddefault}{\updefault}{\color[rgb]{0,0,0}$1$}%
}}}}
\put(2656,-1681){\makebox(0,0)[lb]{\smash{{\SetFigFont{14}{16.8}{\rmdefault}{\mddefault}{\updefault}{\color[rgb]{0,0,0}(b)}%
}}}}
\put(1081,-2941){\makebox(0,0)[lb]{\smash{{\SetFigFont{12}{14.4}{\rmdefault}{\mddefault}{\updefault}{\color[rgb]{0,0,0}$3$}%
}}}}
\put(1846,-2941){\makebox(0,0)[lb]{\smash{{\SetFigFont{12}{14.4}{\rmdefault}{\mddefault}{\updefault}{\color[rgb]{0,0,0}$k$}%
}}}}
\put(1351,-2941){\makebox(0,0)[lb]{\smash{{\SetFigFont{12}{14.4}{\rmdefault}{\mddefault}{\updefault}{\color[rgb]{0,0,0}$4$}%
}}}}
\put(1351,-2176){\makebox(0,0)[lb]{\smash{{\SetFigFont{12}{14.4}{\rmdefault}{\mddefault}{\updefault}{\color[rgb]{0,0,0}$2$}%
}}}}
\put(496,-1861){\makebox(0,0)[lb]{\smash{{\SetFigFont{10}{12.0}{\rmdefault}{\mddefault}{\updefault}{\color[rgb]{0,0,0}type \rmII$'$}%
}}}}
\put(2701,-1861){\makebox(0,0)[lb]{\smash{{\SetFigFont{10}{12.0}{\rmdefault}{\mddefault}{\updefault}{\color[rgb]{0,0,0}type \rmIII$'$}%
}}}}
\put(4681,-2131){\makebox(0,0)[lb]{\smash{{\SetFigFont{12}{14.4}{\rmdefault}{\mddefault}{\updefault}{\color[rgb]{0,0,0}$2$}%
}}}}
\put(451,-1681){\makebox(0,0)[lb]{\smash{{\SetFigFont{14}{16.8}{\rmdefault}{\mddefault}{\updefault}{\color[rgb]{0,0,0}(a)}%
}}}}
\end{picture}%

  \caption{The function $f$ for Case~1 and Case~2 in the proof of \ref{B0min}}\label{Fig2}
\end{center}
\end{figure}

As in the proof of \ref{Btriv} it is enough to show that
$\Gamma_{M_{f}}$ is transitive. Assume $\Gamma_{M_{f}}\models V$ for a
matrix $V=(v_{ij})_{i,j\in A}$, i.e., the rows and colums of $V$ all are of the form
 $\id_{A}=(1,2,3,\dots,k)$, $f=(1,1,2,\dots,2)$ or
$\bc_{i}=(i,i,i,\dots,i)$ ($i\in\{1,\dots,k\}$). 
We have to show that the diagonal $d_{V}:=(v_{11},\dots,v_{mm})$
belongs to $\Gamma_{M_{f}}$. Step by step we reduce the cases to be
checked.

(a) We start with $v_{11}=i\neq 1$ for some $i\in\{2,\dots,k\}$. Then
$\bv_{1,*}=\bc_{i}$  (otherwise $\bv_{1,*}\notin\Gamma_{M_{f}}$), thus,
for each $j\in\{2,\dots,k\}$ we have
$v_{1j}=i$ what implies $\bv_{*,j}=\bc_{i}$. Consequently
$d_{V}=\bc_{i}\in \Gamma_{M_{f}}$ and we are done.

(b) Now we can assume $v_{11}=1$. Then $\bv_{1,*},
\bv_{*,1}\in\{\id_{A},f,\bc_{1}\}$ what implies
$\bv_{2,*}\in\{\id_{A},f,\bc_{1}\}$ and therefore we have
$v_{12},v_{21},v_{22}\in\{1,2\}$.

Let \framebox{$v_{22}=1$}. Then $\bv_{*,2},\bv_{2,*}\in\{f,\bc_{1}\}$ (because $f$ and
$\bc_{1}$ are the only elements of $\Gamma_{M_{f}}$ with value $1$ in
the second component), in particular $v_{2i}\in\{1,2\}$ for all $i$.

If $\bv_{*,i}=\bc_{j}$ is constant for some $i\geq 3$, then
$j\in\{1,2\}$ (because $v_{2i}\in\{1,2\}$) and all rows $\bv_{\ell,*}$
must be equal to $\bc_{j}$ for all $\ell\geq 3$, consequently
$d_{V}=(1,1,j,\dots,j)\in\Gamma_{M_{f}}$ for $j\in\{1,2\}$.

If $\bv_{*,i}=f$ for some $i\geq 3$, then all rows $\bv_{\ell,*}$
must be equal to $f$ for all $\ell\geq 3$ (in no other element of
$\Gamma_{M_{f}}$ appears $2$ at the $i$-th place), consequently
$d_{V}=(1,1,2,\dots,2)\in\Gamma_{M_{f}}$. The same arguments apply for
the cases $\bv_{i,*}\in\{f,\bc_{j}\}$ for some $i\geq 3$ (change the
role of rows and columns). 

Thus it remains to consider the case that
all $\bv_{*,i}$ and $\bv_{i,*}$ ($i\geq 3$) are neither $f$ nor some
$\bc_{j}$. However then all these
columns and rows were equal to $\id_{A}$, but this cannot appear because, e.g.,
$\bv_{3,*}=\id_{A}$ and $\bv_{*,4}=\id_{A}$ would give $v_{34}=4$ and $v_{34}=3$,
respectively, a contradiction. Note that here is used the fact $k\geq 4$.

Now let \framebox{$v_{22}=2$}. Then $\bv_{2,*}\in\{\id_{A},\bc_{2}\}$.

If $\bv_{2,*}=\id_{A}$, then we must have $\bv_{*,j}=\bc_{j}$ for $j\geq 3$, thus
$d_{V}=(1,2,3,\dots,k)\in\Gamma_{M_{f}}$.
If $\bv_{2,*}=\bc_{2}$, then we must have $\bv_{*,1}=\id_{A}$ (recall
$v_{11}=1$). Consequently, $\bv_{j,*}=\bc_{j}$ for $j\geq 3$ and we also
get $d_{V}=(1,2,3\dots,k)\in\Gamma_{M_{f}}$.

\underline{Case 2:} \textit{$f$ is of type \rmIII$'$ but not of type \rmIII},
i.e., $f^{p}=\id_{A}$ for some prime $p$ and the permutation $f$ has
only one cycle of length $p$ but $m$ fixed points $z_{1},\dots,z_{m}$
where $m\geq 2$.

For simplicity let $A=\{0,1,\dots,p-1,z_{1},\dots,z_{m}\}$ where
$0,1,\dots,p-1$ denote the elements of the cycle, i.e., 
$f=(0\,1\,\dots\, p-1)(z_{1})\ldots(z_{m})$, moreover let $k:=p+m=|A|$.
Thus $\Gamma_{M_{f}}$ consists of the $n$-tuples 
$f^{i}=(i,i+1,\dots,i+p-1,z_{1},\dots,z_{m})$
($i\in\Z_{p}=\{0,1,\dots,p-1\}$, all counting in $\Z_{p}$ is done modulo
$p$) and all constants $\bc_{a}=(a,a,\dots,a)$, $a\in A$.

We have to show that $\Gamma_{M_{f}}$ is transitive. Thus let
$\Gamma_{M_{f}}\models V$ where $V$ is an $(k\times k)$-matrix
$V=(v_{ij})_{i,j\in A}$ (here we
enumerate the rows and columns by the elements of $A$).

If $v_{00}=z$ is a fixed point $z\in\{z_{1},\dots,z_{m}\}$ then all
columns and rows of $V$ (as elements of $\Gamma_{M_{f}}$) must be
equal to $\bc_{z}$, thus $d_{V}=(z,\dots,z)\in\Gamma_{M_{f}}$.

Let $v_{00}=i$ for some $i\in\Z_{p}$. Then $v_{*,0}\in\{\bc_{i},f^{i}\}$.

Assume $v_{*,0}=\bc_{i}$. If there exists some row $\bv_{j,*}=f^{i}$
(for some $j\in\Z_{p}$), then $v_{j,z}=z$ and therefore $\bv_{*,z}=z$
for each $z\in\{z_{1},\dots,z_{m}\}$. Thus the last $m$
columns are all different, what implies $\bv_{a,*}=f^{i}$ for all
$a\in A$ (here we need $m\geq 2$).
 Consequently,
 $d_{V}=(i,i+1,\dots,i+p-1,z_{1},\dots,z_{m})\in\Gamma_{M_{f}}$.

Otherwise (if such a row $\bv_{*,j}=f^{i}$ does not exist), all rows
$\bv_{j,*}$ must be equal to $\bc_{i}$ ($j\in\Z_{p}$), what implies
$\bv_{*,z}=\bc_{i}$ for $z\in\{z_{1},\dots,z_{m}\}$, consequently 
$d_{V}=(i,\dots,i,\dots,i)\in\Gamma_{M_{f}}$.
The same arguments apply to the case $v_{0,*}=\bc_{i}$ resulting in
$d_{V}\in\Gamma_{M_{f}}$. 

Thus it remains to consider the case
$\bv_{*,0}=f^{i}$ \textbf{and} $\bv_{0,*}=f^{i}$. However, this case
cannot occur since then $\bv_{z_{1},*}=\bc_{z_{1}}$ and
$\bv_{*,z_{2}}=\bc_{z_{2}}$ leads to the contradiction
$v_{z_{1},z_{2}}=z_{1}$ and $v_{z_{1},z_{2}}=z_{2}$ (note $m\geq 2$).

\textsc{Part 2:} Now we show that there are no more minimal u-closed monoids
than those of type \rmI, \rmII$'$ and \rmIII$'$. There are only the
following two cases (A) and (B) for functions $f$ to be considered for which
$M_{f}$ is $C$-minimal (i.e., satisfies (i)--(iii))
but which are not of type \rmI, \rmII$'$ or \rmIII$'$. We are going to
show that for these $f$ the u-closure $\uclose{M_{f}}$ is not minimal
what will finish the proof of the Theorem.

Case (A): \textit{$f^{2}$ is constant and $|A|=3$.}

There is only one (up to isomorphism) such function $f$ on a
$3$-element set and we use the notation from Figure~\ref{Fig3}(A). Then
$M_{f}=\{\id_{A},f,\bc_{0},\bc_{1},\bc_{2}\}$. Consider the binary mapping
$h$ defined by the following table:
\begin{align*}
  \begin{array}[b]{c|ccc||c}
    h & 0&1&2&\in\Gamma_{M_{f}}\\\hline
    0 & 0&0&0&\bc_{0}\\
    1 & 0&0&1&f\\
    2 & 0&1&2&\id_{A}\\\hline
  \end{array}
\end{align*}
Clearly $h\in M_{f}^{*}$ (as indicated in the last column). Therefore
(cf.~\ref{M8})
$g:=\Delta h\in \uclose{M_{f}}$ where $g$ (see Figure~\ref{Fig3}(b)) is a
function of type \rmI. Thus, by \ref{B0}, we get
$\uclose{M_{g}}=M_{g}\subset \uclose{M_{f}}$, i.e., $\uclose{M_{f}}$ is
not minimal u-closed.

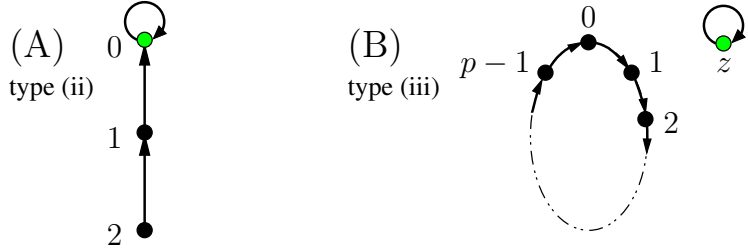
\begin{figure} \begin{center}
\begin{picture}(0,0)%
\includegraphics{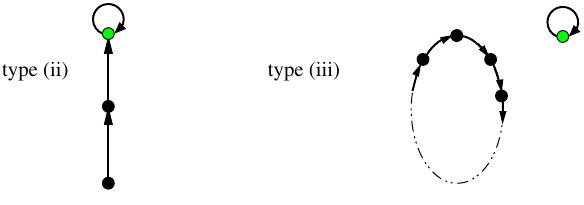}%
\end{picture}%
\setlength{\unitlength}{4144sp}%
\begingroup\makeatletter\ifx\SetFigFont\undefined%
\gdef\SetFigFont#1#2#3#4#5{%
  \reset@font\fontsize{#1}{#2pt}%
  \fontfamily{#3}\fontseries{#4}\fontshape{#5}%
  \selectfont}%
\fi\endgroup%
\begin{picture}(4431,1540)(751,-2875)
\put(2791,-1681){\makebox(0,0)[lb]{\smash{{\SetFigFont{14}{16.8}{\rmdefault}{\mddefault}{\updefault}{\color[rgb]{0,0,0}(B)}%
}}}}
\put(4186,-1501){\makebox(0,0)[lb]{\smash{{\SetFigFont{12}{14.4}{\rmdefault}{\mddefault}{\updefault}{\color[rgb]{0,0,0}$0$}%
}}}}
\put(4996,-1771){\makebox(0,0)[lb]{\smash{{\SetFigFont{12}{14.4}{\rmdefault}{\mddefault}{\updefault}{\color[rgb]{0,0,0}$z$}%
}}}}
\put(4681,-2131){\makebox(0,0)[lb]{\smash{{\SetFigFont{12}{14.4}{\rmdefault}{\mddefault}{\updefault}{\color[rgb]{0,0,0}$2$}%
}}}}
\put(3466,-1771){\makebox(0,0)[lb]{\smash{{\SetFigFont{12}{14.4}{\rmdefault}{\mddefault}{\updefault}{\color[rgb]{0,0,0}$p-1$}%
}}}}
\put(4591,-1771){\makebox(0,0)[lb]{\smash{{\SetFigFont{12}{14.4}{\rmdefault}{\mddefault}{\updefault}{\color[rgb]{0,0,0}$1$}%
}}}}
\put(1351,-2221){\makebox(0,0)[lb]{\smash{{\SetFigFont{12}{14.4}{\rmdefault}{\mddefault}{\updefault}{\color[rgb]{0,0,0}$1$}%
}}}}
\put(1351,-1681){\makebox(0,0)[lb]{\smash{{\SetFigFont{12}{14.4}{\rmdefault}{\mddefault}{\updefault}{\color[rgb]{0,0,0}$0$}%
}}}}
\put(1351,-2806){\makebox(0,0)[lb]{\smash{{\SetFigFont{12}{14.4}{\rmdefault}{\mddefault}{\updefault}{\color[rgb]{0,0,0}$2$}%
}}}}
\put(766,-1681){\makebox(0,0)[lb]{\smash{{\SetFigFont{14}{16.8}{\rmdefault}{\mddefault}{\updefault}{\color[rgb]{0,0,0}(A)}%
}}}}
\end{picture}%

  \caption{The remaining functions of type (ii) and (iii) in the proof
  of \ref{B0min}}\label{Fig3}
\end{center}
\end{figure}

Case (B): \textit{$f^{p}=\id_{A}$, $f$ consists of a single $p$-cycle and has at
most one fixed point.}

For $f$ we use the notation as in Figure~\ref{Fig3}(B),
$A=\{0,1,\dots,p-1,z\}$. All computation in $\Z_{p}=\{0,1,\dots,p-1\}$
is done modulo $p$. If $f$ has no fixed point, $z$ can be ignored in
all what follows. We have $M_{f}=\{\id,f,f^{2},\dots,f^{p-1},\bc_{0},\bc_{1},\dots,\bc_{p-1},\bc_{z}\}$. Consider the binary mapping $h$ defined by the
following table: 
\begin{align*}
  \begin{array}[t]{c|cc @{\hspace*{2ex}\dots}cc|c||c}
    h&0&1&&p-1&z&\in\Gamma_{M_{f}}\\\hline
   0&0&1&&p-1&z&\id_{A}\\
   1&1&2&&0&z&f\\
   2&2&3&&1&z&f^{2}\\
   \vdots&\vdots&\vdots&&\vdots&\vdots&\vdots\\
   p-1&p-1&0&&p-2&z&f^{p-1}\\\hline
   z&z&z&&z&z&\bc_{z}\\\hline
  \end{array}
\end{align*}
Clearly $h\in M^{*}$ (indicated in the last column). Therefore
(cf.~\ref{M8}) $g:=\Delta h\in \uclose{M_{f}}$ and $g$ is the
permutation $g:x\mapsto 2x$ for $x\in Z_{p}$ and $gz=z$. Note that $0$
is an additional fixed point. First we consider the
case that $p\geq 5$. In the group generated by $g$ there
must exist an element $g'$ of prime order $q$ with $q<p$. Since $p\geq
5$, $g$ has either more than one $q$-cycle or at least two fixed points,
i.e., $g'$ is of type \rmIII$'$. Since $g'\in\Sg{g}\subseteq
\uclose{M_{f}}$ we get (with \ref{B0}) $\uclose{M_{g'}}=M_{g'}\subset
\uclose{M_{f}}$, i.e., $\uclose{M_{f}}$ is not minimal u-closed.

It remains to consider the cases $p=2$ and $p=3$. For $p=3$, we get
$g=(0)(12)(z)$ (in cycle notation) if 
there exists a fixed point $z$ what is a function of type \rmIII$'$,
and we can continue as above with $g'$.
Otherwise we have $g=(0)(12)$.  
For $p=2$ there must exist the fixed
point $z$ (since $|A|\geq 3$) and we have 
$f=(01)(z)$, what is a function of the same form as $g$ in case $p=3$
(up to isomorphism). Thus we can continue with $g$.
Take the function
$h'$ given by the table
\begin{align*}
\begin{array}[b]{c|ccc||c}
  h'&0&1&2&\in\Gamma_{M_{g}}\\\hline
  0&0&0&0&\bc_{0}\\
  1&0&1&2&\id_{A}\\
  2&0&2&1&g\\\hline
\end{array}
\end{align*}
Then $h'\in M_{g}^{*}$ (as indicated in the last column) and therefore $g'':=\Delta h'$ belongs to
$\uclose{M_{g}}\subseteq \uclose{M_{f}}$. But $g''$ is a function of
type \rmI{} ($g'0=0$, $g'2=g'1=1$). Thus, as above, $\uclose{M_{g''}}=M_{g''}\subset
\uclose{M_{f}}$, i.e., $\uclose{M_{f}}$ is not minimal u-closed.
\end{proof}

\begin{remark}\label{B3}
Comparing Theorem~\ref{B0min} with the above mentioned results
from \mbox{\cite{JakPR2016}}, we can conclude that there are monoids
$M\leq A^{A}$ which are characterizable by generalized
quasiorders but not by quasiorders, i.e., we have $M=\End\gQuord M$ but
$M\subsetneqq \End\Quord M$ (namely those $M_{f}$ with $f$ of type
\rmII$'$ or \rmIII$'$ but not of type \rmII{} or \rmIII). With other
words, generalized quasiorders are really more powerful than
quasiorders (or congruences).

For $|A|=3$, \textsc{M. Behrisch} (personal communication) computed
all monoids of the form $\End Q$ for $Q\subseteq \gQuord(A)$ and of
the form $\End Q$ for $Q\subseteq\Quord(A)$, their number is 89 and
71, respectively, among all 699 monoids $M\leq A^{A}$.
\end{remark}

\begin{remark}\label{B1}
Let $A=\Z_{k}=\{0,1,\dots,k-1\}$, $k\geq 2$, and let $\gamma_{k}\in A^{A}$ be the
full cycle $\gamma_{k}=(01\dots k{-}1)$, i.e., $\gamma_{k}(x)=x+1$ (all computation is
done modulo k). Consider the monoid $M_{\gamma_{k}}=\Sg[S_{k}]{\gamma_{k}}\cup
C$. It can be shown (unpublished result)
that for the u-closure
$\uclose{M_{\gamma_{k}}}=\uclose{\Sg[S_{k}]{\gamma_{k}}}$ we need only
congruence relations instead of all generalized quasiorders (cf.\
Theorem~\ref{A3}), i.e., we have  $\uclose{M_{\gamma_{k}}}=\End\Con
M_{\gamma_{k}}$. This closure contains much more elements than
$M_{\gamma_{k}}$ (namely, if $k=p_{1}^{m_{1}}\cdot\ldots\cdot p_{n}^{m_{n}}$ is
the decomposition of $k$ into powers of different primes, then
$|\uclose{M_{\gamma_{k}}}|=
    \sum_{i=1}^{n}p_{i}^{\,p_{i}+p_{i}^{2}+\ldots+p_{i}^{m_{i}}}$ ). In
particular, $M_{\gamma_{k}}$ is not u-closed (what was proved, at
least for prime $k=p$, already
with Part II, Case (B), in the proof of
Theorem~\ref{B0min}). 
\end{remark}

\begin{lattices}\label{B2}
For fixed $A$ and fixed arity $m\in\N_{+}$, the set $\gQuorda[m](A,F)$
of all $m$-ary generalized 
quasiorders of an algebra $(A,F)$ forms a lattice with respect to
inclusion (where one can
restrict $F$ to unary mappings because of \ref{A1b}). All these
lattices together also form a lattice, namely 
\begin{align*}
  \cK^{(m)}_{A}:=\{\gQuorda[m](A,F)\mid F\subseteq A^{A}\}.
\end{align*}
For $m=2$ this lattice was investigated in \cite{JakPR2016} (note $\Quord(A,F)=\gQuorda[2](A,F)$).
Due to the Galois connection $\End-\gQuord$ the
lattice $\cK^{(m)}_{A}$ is dually isomorphic to the lattice of all those
u-closed monoids 
$M\leq A^{A}$ which are endomorphism monoids of $m$-ary generalized
quasiorders.

The ``largest'' lattice $\cK^{(k)}_{A}$ with  $k:=|A|$ is isomorphic to the
lattice of all u-closed monoids. 
 With Theorem~\ref{B0min} we also determined the
maximal elements of this lattice $\cK^{(k)}_{A}$, which are of the form 
$\gQuord M_{f}$ with $f$ satisfying one of the conditions \rmI,
\rmII$'$ or \rmIII$'$.

This $\cK^{(k)}_{A}$ contains all $\cK^{(m)}_{A}$ for $m<k$ via an
order embedding. In fact, for $m<n$, there is an order embedding
$\phi^{m}_{n}:\cK^{(m)}_{A}\hookrightarrow \cK^{(n)}_{A}$ given by
$\phi^{m}_{n}(\gQuorda[m](A,F)):=\gQuorda[n](A,\widehat{F})$ with
$\widehat{F}:=\End\gQuorda[m](A,F)$.

Conversely, there is a surjective order preserving map $\psi^{n}_{m}:\cK^{(n)}_{A}\to\cK^{(m)}_{A}$
given by  $\psi^{n}_{m}(\gQuorda[n](A,F)):=\gQuorda[m](A,F)$. This
mapping is well-defined because $\gQuorda[m](A,F)$ is ``contained'' in
$\gQuorda[n](A,F))$ since $\gQuorda[m](A,F)=\{\rho\in\Rela[m](A)\mid A^{n-m}\times\rho\in\gQuorda[n](A,F)\}$
where

\centerline{$A^{n-m}\times\rho=\{(a_{1},\dots,a_{n-m},b_{1},\dots,b_{m})\mid
a_{1},\dots,a_{m}\in A, (b_{1},\dots,b_{m})\in\rho\}$ }
(it is easy to
see that $A^{n-m}\times\rho$ is a generalized quasiorder if and only
if  $\rho$ is). Thus $\rho\mapsto A^{n-m}\times\rho$ is an order
embedding from $\gQuorda[m](A,F)$ into $\gQuorda[n](A,F)$.

\end{lattices}

\section{Concluding remarks}\label{secD}

An algebra $(A,F)$ is called \New{affine complete} if every function
compatible with all congruence relations
of $(A,F)$ is a polynomial 
function, equivalently (for finite $A$), if $\Pol\Con(A,F)$
 is the
clone $\Sg[clone]{F\cup C}$ generated by $F$ and the constants
$C$. With the notation introduced in  \eqref{M1c} (and due to
Remark~\ref{M2X}) we have:
\begin{align*}
  (A,F)\text{ affine complete}&\iff \Sg[clone]{F\cup C}=\Pol\Con(A,F),\\
& \iff \exists Q\subseteq{\Eq(A)}:\\
& \phantom{\iff .}\Sg[clone]{F\cup C}=M^{*} \text{ for } M:=\End Q.
\end{align*}

Instead of equivalence relations we may now consider other relations
which also satisfy the property $\Xi$
(cf.~\ref{MX}). This leads to the notion \New{generalized quasiorder
  complete}, or \New{$\gQuord$-complete} for short,
which can be defined and characterized as follows:
\begin{align*}
  (A,F)\text{ $\gQuord$-complete}:&\iff \Sg[clone]{F\cup C}=\Pol\gQuord(A,F)\\
& \iff \exists Q\subseteq{\gQuord(A)}:\\
& \phantom{\iff .}\Sg[clone]{F\cup C}=M^{*} \text{ for } M:=\End Q\\
& \iff \exists \text{ u-closed } M\leq A^{A}: \Sg[clone]{F\cup C}=M^{*}.
\end{align*}

As an intermediate step one might introduce \New{$\Quord$-complete} algebras
(replacing $\gQuord$ by $\Quord$ above). 

Clearly, affine completeness implies $\gQuord$-completeness (but not
conversely).  Thus it is natural to ask which algebraic properties of
affine complete algebras remain valid for $\gQuord$-complete algebras.
Moreover, what can be said about varieties generated by $\gQuord$-complete algebras?

We recall that a variety $\mathcal{V}$ is called affine
  complete, if all 
algebras $A\in\mathcal{V}$ are affine complete. Similarly, we can define a
\New{$\gQuord$-complete variety} by the property that all its algebras $A\in
\mathcal{V}$ are $\gQuord$-complete. Hence, by our definition,
$\gQuord$-complete 
varieties can be considered a generalization of the affine complete varieties.
It is known that any affine complete variety is congruence
distributive (see e.g. 
\cite{KaaM1997}). There arises the question what are the properties
of $\gQuord$-complete varieties,  could they be still congruence distributive? In
the paper \cite{KaaM1997} also a characterization of affine complete arithmetical
varieties is established (A variety is called arithmetical, if any
algebra in it is congruence distributive \textsl{and} congruence permutable.) Therefore, it is
meaningful to ask if there exists any characterization for $\gQuord$-complete
arithmetical algebras.

We mention some further topics for research:

- Characterize the u-closed monoids which are already given by their 
quasiorders or congruences (cf. Remarks \ref{B3}, \ref{B1}), i.e.,
monoids $M$ with the property $M=\End\gQuord M=\End\Quord M$ or
$M=\End\gQuord M=\End\Con M$.

- Characterize the Galois closures $\gQuord\End
  Q$, cf. \ref{R1}.

- Investigate the lattices $\cK^{(m)}_{A}$ (Remark
  \ref{B2}) and
  their interrelations.

\vfill

\textsc{Acknowledgement.} 
The research of the first author was supported by the Slovak VEGA grant
1/0152/22. The research of the third author was 
carried out as part of the 2020-1.1.2-PIACI-KFI-2020-00165 ``ERPA''
project -- supported by the National Research Development and Innovation Fund of
Hungary.

\vfill

\section*{Remarks by two of the coauthors}

\it\small

In June 2022, a Honorary colloquium on the occasion of Reinhard
P\"oschel's 75th birthday was held in Dresden. There R. Pöschel
presented a talk containing the basics of this article (\cite{JakPR2022}).
The colloquium was organized by
M. Bodirsky and M. Schneider, who at the same time
informed about a forthcoming topical collection of Algebra Universalis, which
will be dedicated to R. P\"oschel. At that time, the full version of
the presented results was not yet written.

We, the co-authors of the results, somehow also would like to
contribute to this honorary commemoration and therefore here -- because we
cannot submit it to the topical collection -- we use the presentation
of our common results as an opportunity
to express our deep respect and gratitude to Reinhard, for his
inventiveness, creativity, energy, and for his kindness.  For more than 16 years
we both have been working successfully together with Reinhard who was
the initiator of many of our joint works.
Our thanks also go to Martin Schneider for his activities.\\
June 2023\hspace*{\fill} Danica Jakub{\'\i}kov\'a-Studenovsk\'a and S\'andor Radeleczki

\rm

\newpage

\def\cprime{$'$} \def\cprime{$'$}

Danica Jakub{\'\i}kov\'a-Studenovsk\'a:~Institute of
  Mathematics, P.J.~\v{S}af\'arik University,~Ko\v sice,\\ \texttt{danica.studenovska@upjs.sk},

Reinhard P{\"o}schel: Institute of Algebra, Technische Universität Dresden,\\ \texttt{reinhard.poeschel@tu-dresden.de},

S\'andor Radeleczki: Institute of Mathematics, University of
    Miskolc,\\ \texttt{sandor.radeleczki@uni-miskolc.hu}.



\end{document}